\documentclass[12pt]{amsart}

\usepackage{amsmath}
\usepackage{mathscinet}
\usepackage{amssymb}  
\usepackage{latexsym} 
\usepackage{comment}
\usepackage{url}
\usepackage{fullpage,url,amssymb,amsmath,amsthm,amsfonts,mathrsfs}
\usepackage[usenames,dvipsnames]{color}
\usepackage[pagebackref = true, colorlinks = true, linkcolor = blue, citecolor = Green]{hyperref}
\usepackage[alphabetic,lite,nobysame]{amsrefs}
\usepackage{enumitem}
\usepackage{amscd}   
\usepackage[all, cmtip]{xy} 
\usepackage{xfrac}
\usepackage{subfigure}

\DeclareFontEncoding{OT2}{}{} 
\newcommand{\textcyr}[1]{%
 {\fontencoding{OT2}\fontfamily{wncyr}\fontseries{m}\fontshape{n}\selectfont #1}}

\usepackage[T2A,T1]{fontenc}
\makeatletter
\def\easycyrsymbol#1{\mathord{\mathchoice
  {\mbox{\fontsize\tf@size\z@\usefont{T2A}{\rmdefault}{m}{n}#1}}
  {\mbox{\fontsize\tf@size\z@\usefont{T2A}{\rmdefault}{m}{n}#1}}
  {\mbox{\fontsize\sf@size\z@\usefont{T2A}{\rmdefault}{m}{n}#1}}
  {\mbox{\fontsize\ssf@size\z@\usefont{T2A}{\rmdefault}{m}{n}#1}}
}}
\makeatother
\newcommand{\Bcyr}{\easycyrsymbol{\CYRB}}
\newcommand{\Sha}{{\mbox{\textcyr{Sh}}}}

\usepackage{color} 

\newcommand{\defi}[1]{\textsf{#1}} 

\def\act#1#2%
  {\mathop{}%
   \mathopen{\vphantom{#2}}^{#1}%
   \kern-3\scriptspace%
   #2}

\newcommand{\Z}{{\mathbb Z}}
\newcommand{\Q}{{\mathbb Q}}
\newcommand{\R}{{\mathbb R}}

\newcommand{\F}{{\mathbb F}}
\newcommand{\A}{{\mathbb A}}
\newcommand{\PP}{{\mathbb P}}
\newcommand{\G}{{\mathbb G}}

\newcommand{\Cbar}{{\overline{C}}}
\newcommand{\kbar}{{\overline{k}}}

\newcommand{\Xbar}{{\overline{X}}}
\newcommand{\Ybar}{{\overline{Y}}}

\newcommand{\calA}{{\mathcal A}}

\newcommand{\calG}{{\mathcal G}}

\newcommand{\calJ}{{\mathcal J}}

\newcommand{\calM}{{\mathcal M}}

\newcommand{\calO}{{\mathcal O}}

\newcommand{\calX}{{\mathcal X}}
\newcommand{\calY}{{\mathcal Y}}

\newcommand{\frakL}{{\mathfrak L}}

\newcommand{\res}{\textup{res}}

\DeclareMathOperator{\unr}{unr}
\DeclareMathOperator{\Sel}{Sel}

\DeclareMathOperator{\End}{End}
\DeclareMathOperator{\Hom}{Hom}
\DeclareMathOperator{\disc}{disc}

\DeclareMathOperator{\ord}{ord}

\DeclareMathOperator{\Cor}{Cor}
\DeclareMathOperator{\Res}{Res}
\DeclareMathOperator{\Br}{Br}

\DeclareMathOperator{\Div}{Div}

\DeclareMathOperator{\Pic}{Pic}

\DeclareMathOperator{\Jac}{Jac}
\DeclareMathOperator{\HH}{H}

\DeclareMathOperator{\Spec}{Spec}

\DeclareMathOperator{\inv}{inv}

\newcommand{\val}{\operatorname{val}}

\newtheorem{Theorem}{Theorem}[section]
\newtheorem{Lemma}[Theorem]{Lemma}
\newtheorem{Proposition}[Theorem]{Proposition}

\newtheorem{Definition}[Theorem]{Definition}

\newtheorem{Remark}[Theorem]{Remark}
\newtheorem{Algorithm}[Theorem]{Algorithm}

\numberwithin{equation}{section}

\begin{document}

\title{Brauer-Manin obstructions on hyperelliptic curves}

\author{Brendan Creutz}

\author{Duttatrey Nath Srivastava}

\maketitle
\begin{abstract}
We describe a practical algorithm for computing Brauer-Manin obstructions to the existence of rational points on hyperelliptic curves defined over number fields. This offers advantages over descent based methods in that its correctness does not rely on rigorous class and unit group computations of large degree number fields. We report on experiments showing it to be a very effective tool for deciding existence of rational points: Among a random samples of curves over $\Q$ of genus at least $5$ we were able to decide existence of rational points for over $99\%$ of curves. We also demonstrate its effectiveness for high genus curves, giving an example of a genus $50$ hyperelliptic curve with a Brauer-Manin obstrution to the Hasse Principle. The main theoretical development allowing for this algorithm is an extension of the descent theory for abelian torsors to a framework of torsors with restricted ramification.
\end{abstract}

\section{Introduction}

This paper is concerned with the problem of deciding the existence of a rational point on a algebraic curve $C$ defined over a number field $k$. For curves of genus $0$ it has long been known that the existence of $k$-rational points is decidable, as it is equivalent to the existence of points defined over the local fields containing $k$. For curves of positive genus the local-global principle can fail and there is no proven algorithm to decide if the set $C(k)$ of $k$-rational points is nonempty. However, it has been conjectured that the nonexistence of rational points on curves can always be explained by a Brauer-Manin obstruction \cite{Poonen,Stoll} (This was first posed as a question in \cite{TorsorsAndRationalPoints}). If the conjecture is true, then the existence of rational points on curves is decidable because searching for points by day and searching for obstructions by night must eventually produce an answer. 

We describe a practical algorithm to compute the obstruction coming from the elements of order $2$ in the Brauer group of a hyperelliptic curve.
We have implemented the algorithm in the Magma computational algebra system \cite{magma} and find it performs well for hyperelliptic curves of genus $g \le 10$ given by equations with moderately sized coefficients. We also find that it is a very effective tool for deciding existence of rational points. In a sample of genus $5$ curves over $\Q$ drawn by choosing the coefficients of a defining polynomial uniformly at random from integers of size $\le 100$, the algorithm demonstrated a (nonlocal) Brauer-Manin obstruction for over two thirds of the curves, enabling us to decide on the existence of points for $99.6\%$ of the curves considered. For curves of genus $10$ with coefficients of size absolutely bounded by $10$, we were able to decide existence of rational points for all of the curves in our sample. Such a level of success is in line with the recent result of Bhargava \cite{Bhargava,BGW} that a density approaching $100\%$ of hyperelliptic curves of genus $g$ have an obstruction coming from the $2$-torsion subgroup $\Br(C)[2]$ of the Brauer group of the curve (or a local obstruction).

Our approach requires that we can identify and explicitly represent the correct elements in the infinite group $\Br(C)[2]$ to be used in the computation, and herein lies the main theoretical novelty of the paper. In Section~\ref{sec:generalities} we extend the descent theory for torsors under finite abelian group schemes described in \cite[Chapter 6]{TorsorsAndRationalPoints} building on \cite{ColliotTheleneSansuc} to handle torsors unramified outside a given set of places $S$ of the number field (the original case being when $S$ consists of all places). This enables us to prove that the obstruction coming from $\Br(C)[2]$ is equivalent to that coming from the `unramified outside $S$ subgroup' of $\Br(C)[2]$ for a finite set $S$ (See Theorem~\ref{thm:BrlambdaS} for the precise statement, which applies more generally to any smooth projective and geometrically integral variety). We believe this development will be of interest in its own right. Its relevance in the present context is that it allows us to bound the running time of our algorithm a priori because (modulo constant algebras) the unramified outside $S$ subgroup is finite.

To write down explicit elements of $\Br(C)[2]$ that are unramified outside $S$, we make use of \cite{CreutzViray}, which gives an explicit construction of Brauer classes in a certain subgroup which we denote by $\Br_{\Upsilon}(C)[2] \subset \Br(C)[2]$. In Section~\ref{sec:4} we show, somewhat surprisingly, that the obstruction coming from $\Br(C)[2]$ is equivalent to that coming from $\Br_\Upsilon(C)[2]$. The proof of this fact relies on an interpretation of the elements of the fake Selmer set in \cite{BruinStoll} in terms of torsors that are not geometrically connected. While the result here is specific to the situation of hyperelliptic curves, the idea may prove useful in understanding the connection between Brauer-Manin obstructions and the `fake descents' described in \cite{PoonenSchaefer,BPS,GeneralizedJacs} and elsewhere.

\subsection{Comparison with other methods}
It is known that the obstruction coming from $\Br(C)[2]$ is equivalent to the two-cover descent obstruction described in detail in \cite{BruinStoll} building on \cite{BruinFlynn}. Building on this, our approach offers significant advantages. To explain, let us compare the methods. The descent algorithm first computes a finite set of two-coverings with the property that it contains all locally solvable coverings and then carries out local computations to find the subset of them that are actually locally solvable.  The result is a partition of the set of adelic points surviving $2$-descent in the form
\[
	C(k) \subset C(\A_k)^{2\textup{-desc}} = \bigcup_{(Y,\pi) \in \Sel^2(C/k)} \pi(Y(\A_k))\,.
\]
The first step requires class and unit group information in number fields of degree $O(g)$. While there are subexponential algorithms whose correctness is conditional on the generalized Riemann hypothesis, rigorous computation is usually infeasible for $g > 1$. If this is not computed rigorously, then there is no guarantee that the resulting set contains all of the $k$-rational points.

By way of contrast, the set of adelic points orthogonal to a subgroup $B \subset \Br(C)$ is an intersection\,,
\[
	C(k) \subset C(\A_k)^{B} = \bigcap_{\calA \in B} C(\A_k)^{\calA}\,,
\]
where $C(\A_k)^\calA$ is the set of adelic points that are orthogonal to $\calA$. To compute this one enumerates the elements $\calA_1,\calA_2,\ldots$ of $B$ and at the $n$th step computes the intersection of the first $n$. At every step of the process one has a set $X_n = C(\A_k)^{\{\calA_1,\dots,\calA_n\}}$ which is guaranteed (unconditionally) to contain $C(k)$. This allows us to give an algorithm with similar complexity to the conditional descent algorithm, but whose output is rigorous (See Theorem~\ref{thm:S1} and Algorithm~\ref{alg1}). Moreover, there are situations in which even the conditional descent algorithm is infeasible but we are able to produce a Brauer class and check that it gives an obstruction. In Proposition~\ref{prop:g50} we give an example of a genus $50$ hyperelliptic curve over $\Q$ with a Brauer-Manin obstruction to the local-global principle. Using the algorithm of \cite{BruinStoll} this would require class and unit group computations in a number field of degree $102$ which is completely infeasible even assuming GRH.

There are other methods for computing the set of rational points on general curves including $p$-adic methods based on ideas of Chabauty \cite{Chabauty} (See \cite{PoonenChab} for recent developments in this direction) as well as the Mordell-Weil Sieve \cite{Scharashkin,Flynn,BruinStollMWS}. These require an explicit embedding of the curve in its Jacobian given either by a rational point (in which case existence of points is already decided) or by a rational divisor of degree $1$. Deciding existence of such an embedding is equivalent to deciding existence of rational points on the torsor $\Pic^1_C$ parameterizing divisor classes of degree $1$ on $C$. In practice this can be approached using descent algorithms as described in \cite{CreutzANTSX} or \cite[Section 6]{CreutzViray} or via the Brauer-Manin obstruction using the techniques developed in this paper, with the latter having advantages over the former similar to those described above in the context of points on curves. It should also be noted that these methods require knowledge of the set of rational points on the Jacobian, which typically requires the use of descent based techniques relying on GRH for their practical implementation.

\subsection{Outline of the paper}
Section~\ref{sec:Notation} sets some notation and definitions used throughout the paper. Sections~\ref{sec:generalities} and~\ref{sec:4} are, as noted above, devoted to developing the required connections between the descent and Brauer-Manin obstructions to enable the algorithm. 
Section~\ref{sec:4} also recalls the explicit description of Brauer classes given in \cite{CreutzViray} and provides some further details regarding the evaluation of these Brauer classes at points on the curve. Details of the algorithm are given in Section~\ref{sec:algorithm}. In Section~\ref{sec:Data} we provide the results of experiments with the algorithm as well as some specific examples.

\subsection{Acknowledgements} Both authors were supported by the Marsden Fund Council administered by the Royal Society of New Zealand. We thank Bianca Viray for initial discussions on the possibility of using the results of \cite{CreutzViray} to compute Brauer-Manin obstructions along the lines described here as well as for comments on an earlier draft of the paper.

\section{Notation}\label{sec:Notation}

Throughout the paper $k$ will denote a field of characteristic $0$. When $k$ is a number field we use $\Omega_k$ to denote the set of places of $k$, $S \subset \Omega_k$ to denote a subset containing all archimedean primes and $\calO_{k,S} \subset k$ to denote the ring of $S$-integers. Given $v \in \Omega_k$ we use $k_v$ to denote the completion of $k$ at $v$ and $\calO_v$ to denote the ring of integers in $k_v$.

Let $M$ be a finite \'etale abelian $k$-group scheme whose order is a unit in $\Spec(\calO_{k,S})$. Given $v \in \Omega_k$ we say that an element $\xi \in \HH^1(k_v,M)$ (or in $\HH^1(k,M)$) is \defi{unramified} (at $v$) if it lies in the kernel of the restriction map to $\HH^1(k_v^{\unr},M)$, where $k_v^{\unr}$ is the maximal unramified extension of $k_v$. For any $v \in \Omega_k$, the cup product induces a canonical nondegenerate pairing \cite[Corollary I.2.3]{ADT}
\begin{equation}\label{eq:localpairing}
	\HH^1(k_v,M) \times \HH^1(k_v,M^\vee) \to \Q/\Z\,,
\end{equation}
where $M^\vee = \Hom_k(M,\G_m)$ is the Cartier dual of $M$. If $M(k_v^{\unr}) = M(\kbar)$, then the unramified subgroups are exact annihilators by \cite[Theorem I.2.6]{ADT}. In general, the exact annihilator of the unramified subgroup of $\HH^1(k_v,M^\vee)$ may be larger than the unramified subgroup of $\HH^1(k_v,M)$.

We define $\HH^1_S(k,M)$ to be the subgroup of elements of $\HH^1(k,M)$ that are orthogonal to the unramified subgroup of $\HH^1(k_v,M^\vee)$ for all $v \not\in S$. When $M$ spreads out to a smooth group scheme $\calM$ over $\Spec(\calO_{k,S})$, then $\HH^1_S(k,M)$ consists of those elements that are unramified outside $S$ and $\HH^1_S(k,M)$ can be identified with the \'etale cohomology group $\HH^1(\calO_{k,S},\calM)$ (see \cite[Proposition II.2.9]{ADT} and \cite[Section 6.5.7]{PoonenRatP}).
We define
\begin{align*}
	\Sha^i_S(k,M) &:= \ker\left( \HH^i_S(k,M)\to \prod_{v \in S}\HH^i(k_v,M)\right)\,.
\end{align*}
When $S = \Omega_k$ we abbreviate $\Sha^i(k,M) := \Sha^i_{\Omega_k}(k,M)$. When $M$ spreads out to a smooth group scheme over $\Spec(\calO_{k,S})$, our $\Sha^1_S(k,M)$ agrees with that of \cite[Theorem I.4.10]{ADT}, from which it follows that $\Sha^1_S(k,M)$ is finite for any $S$. Indeed, we can enlarge $S$ to $S'$ so that $M$ spreads out and use that $\Sha^1_S(k,M) \subset \Sha^1_{S'}(k,M)$.

The Brauer group of a scheme $X$ is defined as the \'etale cohomology group $\Br(X) := \HH^2(X,\G_m)$. When $X$ is a variety over a field $k$, we use $\Br_0(X)$ to denote the image of the natural map $\Br(k) \to \Br(X)$ and $\Br_1(X)$ to denote the kernel of the natural map $\Br(X) \to \Br(\Xbar)$, where $\Xbar = X \times_{\Spec(k)}\Spec(\kbar)$ is the base change of $X$ to an algebraic closure of $k$. For a commutative ring $R$ we define $\Br(R) := \HH^2(\Spec(R),\G_m)$. 

\section{A computable description of the Brauer-Manin obstruction}\label{sec:generalities}

Throughout this section $k$ will denote a number field. Suppose $X$ is a smooth, projective and geometrically integral variety over $k$, $G$ is a finite \'etale abelian group scheme over $k$ and $M = G^\vee := \Hom(G,\G_m)$. Let $\lambda:M(\kbar) \to \Pic(\Xbar)$ be a morphism of Galois modules. We consider $X$-torsors under $G$ of type $\lambda$. For example, if $X$ is a hyperelliptic curve with Jacobian $J$ and $\lambda : J[2](\kbar) = \Pic^0(\Xbar)[2] \subset \Pic(\Xbar)$ is the inclusion map, then $X$-torsors of type $\lambda$ are the two-coverings considered in \cite{BruinStoll}. For the precise definition see~\cite[Definition 2.3.2]{TorsorsAndRationalPoints}. 

Let $r : \Br_1(X) \to \HH^1(k,\Pic(\Xbar))$ be the canonical map from the Hochshild-Serre spectral sequence $\HH^p(k,\HH^q(\Xbar,\G_m)) \Rightarrow \HH^{p+q}(X,\G_m)$ \cite[(2.23)]{TorsorsAndRationalPoints}. Then any subgroup $H \subset \HH^1(k,M)$ gives rise to the subgroup $r^{-1}(\lambda_*(H)) \subset \Br_1(X)$. In particular, for any $S \subset \Omega_k$ containing all archimedean primes, the subgroups $\Sha^1_{S}(k,M) \subset \HH^1_{S}(k,M) \subset \HH^1(k,M)$ defined in Section~\ref{sec:Notation} determine subgroups
\begin{align*}
	\Bcyr_{\lambda,S}(X) &:=  r^{-1}\left(\lambda^*\left( \Sha^1_{S}(k,M)\right) \right)\,,\\
	 \Br_{\lambda,S}(X)  &:=  r^{-1}\left(\lambda^*\left( \HH^1_{S}(k,M) \right)\right)\,,\text{ and}\\
	\Br_\lambda(X) &:=  r^{-1}\left(\lambda^*\left(\HH^1(k,M)\right) \right)\,.
\end{align*}
These give a filtration of the algebraic Brauer group,
\[
	 \Bcyr_{\lambda,S}(X) \subset \Br_{\lambda,S}(X) \subset \Br_\lambda(X) \subset \Br_1(X)\,.
\]
We note that $\Br_{\lambda,S}(X)/\Br_0(X)$ is finite when $S$ is finite by \cite[Corollary I.4.15]{ADT}. The following theorem gives a computable description of the obstruction coming from the infinite group $\Br_\lambda(X)/\Br_0(X)$.

\begin{Theorem}\label{thm:BrlambdaS}
	Given $X$ and $\lambda$ there is an explicitly computable finite set of primes $S = S(X,\lambda)$ such that the images of $X(\A_k)^{\Br_\lambda(X)}$ and $X(\A_k)^{\Br_{\lambda,S}(X)}$ in $\prod_{v \in S}X(k_v)$ coincide. In particular, 
	\[
	X(\A_k)^{\Br_\lambda(X)} = \emptyset \quad \Leftrightarrow \quad X(\A_k)^{\Br_{\lambda,S}(X)} = \emptyset\,.
	\]
\end{Theorem}

The proof will be completed at the end of this section. It shows that we may take $S$ to be any subset of $\Omega_k$ containing all archimedean primes such that $X$ spreads out to a smooth proper scheme over $\Spec(\calO_{k,S})$, $\lambda$ spreads out to a morphism of smooth proper group schemes and $S$ contains 
\begin{itemize}
	\item all primes of residue cardinality below an explicit bound coming from the generalized Weil conjectures (which depends only on the Betti numbers in the cohomology of $\Ybar = Y \times_k\kbar$ where $Y$ is a torsor of type $\lambda$), and
	\item enough primes to ensure a certain technical hypothesis holds in the case that $\lambda$ is not injective (in which case the torsors of type $\lambda$ are not geometrically connected).
\end{itemize}

The set just described is not the minimal $S$ for which the theorem holds. We have taken care to state and prove the lemmas used in the proof for sets $S$ smaller than that above where possible, even though this introduces a number of fairly technical points that could otherwise be avoided. Our reason for doing so is that we are ultimately interested in practical computations, in which case it is best to take $S$ as small as possible. These lemmas will be used also in the proof of Theorem~\ref{thm:S1}, which is a version of Theorem~\ref{thm:BrlambdaS} for hyperelliptic curves allowing $S$ to omit odd primes where the reduction has a simple node. Similarly, there are practical reasons requiring us to deal with torsors that are not geometrically connected (corresponding to elements in the `fake Selmer set' of \cite{BruinStoll,PoonenSchaefer} -- see Remark~\ref{rmk:fakeSel}), so working in this generality will pay off.

\begin{Definition}\label{def:unrtorsor}
	We say that a torsor $Y \to X_{k_v}$ under $G$ is \defi{unramified} if the map
	\[
		X(k_v) \ni x_v \mapsto Y_{x_v} \in \HH^1(k_v,G)\,
	\]
	sending a point $x_v$ to the class of the fiber above it has its image contained in the unramified subgroup. We say that an $X$-torsor under $G$ is \defi{unramified at $v \in \Omega_k$} (resp., \defi{unramified outside $S \subset \Omega_k$}) if its base change to $\Spec(k_v)$ is unramified (resp., for all $v \not\in S$).
\end{Definition}

\begin{Lemma}\label{lem:spreading}
	Suppose $X(k_v) \ne \emptyset$ and $X_{k_v}$ spreads out to a smooth projective scheme $\calX_v \to \Spec(\calO_{v})$ and $\lambda$ spreads out to a morphism $\calM_v \to \Pic_{\calX_v/\Spec(\calO_{v})}$ of smooth proper group schemes over $\Spec(\calO_{v})$. Then
	\begin{enumerate}
		\item\label{part1} There exists an unramified $X_{k_v}$-torsor of type $\lambda$.
		\item\label{part2} If $Y \to X_{k_v}$ is a torsor of type $\lambda$ which spreads out to an $\calX_v$-torsor under $\calG_v = \calM_v^\vee$, then $Y \to X_{k_v}$ is unramified.
	\end{enumerate}
\end{Lemma}

\begin{proof}
	Since $X_{k_v}(k_v) = \calX_v(\calO_v) \ne \emptyset$, the type map fits into the following commutative diagram with exact rows,
	\[
		\xymatrix{
			0 \ar[r]& \HH^1(k_v,G) \ar[r]& \HH^1(X_v,G) \ar[r]^-{\textrm{type}} & \Hom(M,\Pic_{X_v}) \ar[r] & 0\\
			0 \ar[r]& \HH^1(\calO_v,\calG_v) \ar[r]\ar[u] & \HH^1(\calX_v,\calG_v) \ar[r]\ar[u] & \Hom(\calM,\Pic_{\calX_v/\calO_v}) \ar[u] \ar[r]& 0
		}
	\]
	where the rows come from the Leray spectral sequence as in \cite[Corollary 2.3.9]{TorsorsAndRationalPoints} (see also \cite[Proposition 3.2]{Antei} for an alternative construction of the bottom row) and the vertical arrows arise from taking generic fibers. As noted in the discussion at the bottom of \cite[p. 33]{TorsorsAndRationalPoints}, the existence of a $k_v$-rational point on $X_v$ implies surjectivity of the type map. As $\calX_v(\calO_v) \ne \emptyset$, the same argument gives surjectivity of the corresponding map over $\Spec(\calO_v)$ (see \cite[Proposition 3.2]{Antei}). The assumption that $\lambda$ spreads out implies (using the diagram) that there is an $X_{k_v}$-torsor of type $\lambda$ which spreads out as in~\eqref{part2}. So,~\eqref{part2} $\Rightarrow$ \eqref{part1}. 
	
	To prove~\eqref{part2}, note that the evaluation map $X(k_v) \to \HH^1(k_v,G)$ factors through the evaluation map of the $\calX_v$-torsor, so its image lies in the image of $\HH^1(\calO_v,\calG_v) \to \HH^1(k_v,G)$ which is the unramified subgroup.	
\end{proof}

\begin{Lemma}\label{lem:BrlambdaS}
	Suppose that $\pi: Y \to X$ is a torsor of type $\lambda$ which is unramified outside $S$. Then 
	\[
		X(\A_k)^{\Br_{\lambda,S}(X)} = \left\{ (x_v) \in X(\A_k)\;:\; \exists\tau \in \HH^1_S(k,G) \text{ such that $\forall v \in S$, }  x_v \in \pi^\tau(Y^\tau(k_v)) \right\}.
	\]
\end{Lemma}

\begin{proof}	
	For some $S'$ containing $S$ such that $S' - S$ is finite, $G$ spreads out to a smooth group scheme $\calG$ over $\Spec(\calO_{k,S'})$ and $Y \to X$ spreads out to a torsor $\calY \to \calX$ under $\calG$. Then $M$ spreads out to a smooth group scheme $\calM$ dual to $\calG$ and $\HH^1_{S'}(k,M) = \HH^1(\calO_{k,S'},\calM)$ and similarly for $\calG$. Moreover, $\HH^1_{S}(k,M) = \ker\left (\HH^1(\calO_{k,S'},\calM) \to \prod_{v \in S'-S}\HH^1(k_v,M)/U_v^\perp\right)$, where $U_v^\perp$ denotes the annihilator of the unramified subgroup $U_v \subset \HH^1(k_v,G)$ under the pairing~\eqref{eq:localpairing}. In this situation the generalized Poitou-Tate exact sequence \cite[Theorem 6.2]{Cesnavicius} (see also~\cite[Proposition 4.6]{CreutzGW} for a proof in the case of finite $S'$ that only uses Galois cohomology) gives an exact sequence
	\[
		\xymatrix{
		\HH^1_{S'}(k,G) \ar[r]& \prod'_{v \in S} \HH^1(k_v,G) \times \prod_{v \in S' - S} \frac{\HH^1(k_v,G)}{U_v} \ar[r]& \HH^1_{S}(k,M)^*\\
		& (\tau_v)_{v\in S'} \ar@{|->}[r] & \left[ \alpha \mapsto \Sigma_{v \in S'} \inv_v(\tau_v \cup \alpha_v)\right],
		}
	\]
	where the cup product is induced by the pairing in~\eqref{eq:localpairing}.
	
	For $\beta \in \Br_{\lambda,S}(X)$ corresponding to $\alpha \in  \HH^1_{S}(k,M)$ and $(x_v) \in X(\A_k)$ we have, as in \cite[(6.8) on p. 121]{TorsorsAndRationalPoints},
	\begin{equation}\label{brtorsor}
		\Sigma_{v \in \Omega_k} \inv_v(\beta(x_v)) = \Sigma_{v \in \Omega_k} \inv_v(Y_{x_v} \cup \alpha_v) = \Sigma_{v \in S} \inv_v(Y_{x_v} \cup \alpha_v)\,,
	\end{equation}
	where the final equality follows from the fact that, for $v \not\in S$, $Y_{x_v}$ is unramified and $\alpha_v := \res_{k_v/k}(\alpha)$ is orthogonal to the unramified subgroup. Exactness of the Poitou-Tate sequence above shows that there exists $\tau \in \HH^1_{S'}(k,G)$ with the same image as $(Y_{x_v})_{v \in S'}$ if and only if $(x_v) \in X(\A_k)^{\Br_{\lambda,S}(X)}$. Note that for any such $\tau$ we have $\tau_v = Y_{x_v}$ for $v \in S$ and $\tau \in \HH^1_S(k,G)$ because $Y_{x_v} \in U_v$ for $v \in S' - S$. By definition, $Y_{x_v} = \tau_v$ is equivalent to $x_v \in \pi^{\tau_v}(Y^{\tau_v}(k_v))$, so this proves the lemma. 
\end{proof}

The next lemma characterizes the existence of unramified outside $S$ torsors of type $\lambda$ in terms of the Brauer-Manin obstruction.

\begin{Lemma}\label{lem:existenceLambda}
	Suppose that $X(\A_k) \ne \emptyset$ and that for every $v \not\in S$ there is an unramified $X_{k_v}$-torsor of type $\lambda$ defined over $k_v$. Then there exists an $X$-torsor of type $\lambda$ unramified outside $S$ if and only if $X(\A_k)^{\Bcyr_{\lambda,S}(X)} \ne \emptyset$.
\end{Lemma}

\begin{Remark}\label{rmk:m}\hfill
	\begin{enumerate}
		\item By Lemma~\ref{lem:spreading} the hypothesis is satisfied if $X$ spreads out to a smooth proper scheme and $\lambda$ spreads out to a morphism of smooth proper group schemes over $\Spec(\calO_{k,S})$. 
		\item The hypothesis is satisfied for $S = \Omega_k$, in which case the result is a crucial step in establishing the descent theory for abelian torsors (See (1) at the top of page 115 of \cite{TorsorsAndRationalPoints}). In this case the result holds more generally for $M = G^\vee$ of finite type. For general $S$ we must restrict to finite $M$ as the proof relies on the Poitou-Tate exact sequence which requires that $M$ is finite or that $S$ is cofinite.
	\end{enumerate}
\end{Remark}

\begin{proof}
	First suppose there exists a torsor $Y \to X$ of type $\lambda$ unramified outside $S$. For any $v \in \Omega_k$, any $x_v \in X(k_v)$ and any $\alpha \in \Sha^1_{S}(k,M)$ we have $Y_{x_v} \cup \alpha_v = 0$. For $v \in S$ this is because $\alpha_v = 0$ and for $v \notin S$ this is because $Y_{x_v}$ is unramified and hence orthogonal to $\alpha_v$. If $\beta\in \Bcyr_{\lambda,S}(X)$ corresponds to $\alpha$, then~\eqref{brtorsor} shows that $X(\A_k)^\beta = X(\A_k)$. This proves one direction of the lemma.
	
	For the converse, suppose $X(\A_k)^{\Bcyr_{\lambda,S}(X)} \ne \emptyset$. Then $X(\A_k)^{\Bcyr_\lambda(X)} \ne \emptyset$ and so there exists a torsor $Y \to X$ of type $\lambda$ (cf. Remark~\ref{rmk:m}). For some $S' \supset S$ with $S' - S$ finite, $G$ spreads out to a smooth proper group scheme $\calG$ over $\Spec(\calO_{k,S'})$ and $Y \to X$ spreads out to a torsor $\calY \to \calX$ under $\calG$ with $\calX$ smooth and proper over $\Spec(\calO_{k,S'})$. By Lemma~\ref{lem:spreading}, $Y \to X$ is unramified outside $S'$. 
	
	The Cartier dual $\calM$ of $\calG$ has generic fiber $M$ and 
	\[
		\Sha^1_{S}(k,M) = \ker\left( \HH^1(\calO_{k,S'},\calM) \to \prod_{v \in S} \HH^1(k_v,M) \times \prod_{v \in S'-S}\HH^1(k_v,M)/U_v^\perp\right)\,,
	\]
	where $U_v \subset \HH^1(k_v,G)$ is the unramified subgroup and $U_v^\perp \subset \HH^1(k_v,M)$ is its annihilator under the pairing~\eqref{eq:localpairing}. The generalized Poitou-Tate sequence \cite[Theorem 6.2]{Cesnavicius} gives
	\[	
		\HH^1_{S'}(k,G) \to \prod_{v \in S' - S}\HH^1(k_v,G)/U_v \to \Sha^1_{S}(k,M)^*\,.
	\]
	(We have omitted the factors at $v \in S$ in the central term as they are $\HH^1(k_v,G)/0^\perp = \HH^1(k_v,G)/\HH^1(k_v,G) = 0$.)
	Let $(x_v) \in X(\A_k) = X(\A_k)^{\Bcyr_{\lambda,S}(X)}$. Arguing as in the proof of Lemma~\ref{lem:BrlambdaS} using~\eqref{brtorsor} we obtain $\tau \in \HH^1(\calO_{k,S'},\calG)$ mapping to the image of $(Y_{x_v})_{v \in {S'-S}}$.

	We claim that the twist of $Y \to X$ by $\tau$ is unramified outside $S$. It is unramified outside $S'$ since both $Y \to X$ and $\tau$ are unramified outside $S'$. For $v \in S' - S$, we have $(Y^\tau)_{x_v} = Y_{x_v} - \tau_v \in U_v$. The image of the evaluation map $X(k_v) \to \HH^1(k_v,G)$ given by $Y^\tau\to X$ therefore intersects the unramified subgroup. By assumption there exists an unramified torsor $Y_v \to X_{k_v}$ of type $\lambda$ and the base change of $Y^\tau \to X$ to $k_v$ is a twist of this unramified torsor. It follows that the image of the evaluation map $X(k_v) \to \HH^1(k_v,G)$ given by $Y^\tau\to X$ lies in a coset of the unramified subgroup. As this coset has nonempty intersection with the unramified subgroup, it must be the unramified subgroup. So $Y \to X$ is unramified at $v \in S' - S$ as well.
\end{proof}

\begin{Lemma}\label{lem:Sassumptions}
	Suppose $S \subset \Omega_k$ is such that 
	\begin{enumerate}
		\item\label{assumption1} there exists a torsor $\pi:Y \to X$ of type $\lambda$ unramified outside $S$,
		\item\label{assumption2} for all torsors $\pi : Y \to X$ of type $\lambda$ which are unramified outside $S$ and locally soluble at all primes $v \in S$, we have $Y(\A_k) \ne \emptyset$.
	\end{enumerate}
	Then the conclusion of Theorem~\ref{thm:BrlambdaS} holds.
\end{Lemma}

\begin{proof}
	Since $ X(\A_k)^{\Br_\lambda(X)}\subset X(\A_k)^{\Br_{\lambda,S}(X)}$ it suffices to prove that
	\[
		\rho_S\left(X(\A_k)^{\Br_{\lambda,S}(X)} \right) \subset  \rho_S\left(X(\A_k)^{\Br_\lambda(X)}\right)\,,
	\]
	where $\rho_S$ denotes the projection map $\rho_S : X(\A_k) = \prod_{v \in \Omega_k} X(k_v) \to \prod_{v\in S} X(k_v)$. So let $(x_v)_{v\in S} \in \rho_S\left(X(\A_k)^{\Br_{\lambda,S}(X)} \right)$. By Lemma~\ref{lem:BrlambdaS} there exists $\tau \in \HH^1_S(k,G)$ such that $x_v \in \pi^\tau(Y^\tau(k_v))$ for all $v \in S$. By assumption~\eqref{assumption2}, $Y^\tau(\A_k) \ne \emptyset$. In particular, there exists $y = (y_v)_{v\in \Omega_k} \in Y^\tau(\A_k)$ such that $\pi^\tau(y_v) = x_v$ for $v \in S$. In other words, $(x_v)_{v \in S} = \rho_S(\pi^\tau(y))$. On the other hand
	\[
		\pi^\tau(y) \in \bigcup_{\tau \in \HH^1_S(k,G)} \pi^\tau(Y^\tau(\A_k)) \subset \bigcup_{\tau \in \HH^1(k,G)} \pi^\tau(Y^\tau(\A_k)) = X(\A_k)^{\Br_\lambda(X)}\,,
	\]
	where the final equality is by \cite[Theorem 6.1.2]{TorsorsAndRationalPoints}. So $(x_v)_{v \in S} \in \rho_S\left(X(\A_k)^{\Br_\lambda(X)}\right)$ as required.
\end{proof}

\begin{proof}[Proof of Theorem~\ref{thm:BrlambdaS}]
	Suppose the type map $\lambda : M \to \Pic(\Xbar)$ factors as $M \twoheadrightarrow M_0 \hookrightarrow \Pic(\Xbar)$. Then $G_0 := M_0^\vee \subset G$ and $\Ybar \to \Xbar$ is a disjoint union of torsors $\Ybar_0 \to \Xbar$ under $\overline{G}_0$ (which are torsors of type $\lambda_0 : M_0 \to \Pic(\Xbar)$). The Weil conjectures (or the earlier result of Lang-Weil) give an explicitly computable bound $B$ depending only on $\Ybar_0 \to \Xbar$ such that if $v \in \Omega_k$ has residue cardinality larger than $B$ and $\calY_0 \to \Spec(\calO_v)$ is smooth and has geometric generic fiber isomorphic to $\Ybar_0$, then $\calY_0(\F_v) \ne \emptyset$. By Hensel's lemma $\calY_0(k_v) \ne \emptyset$ as well. 

	There is a finite set of primes $S \subset \Omega_k$ containing all archimedean primes such that $X$ spreads out to smooth proper scheme $\calX$, $G$ and $M$ spread out to Cartier dual smooth group schemes $\calG$ and $\calM$ and $\lambda$ spreads out to a morphism $\calM \to \Pic_{\calX/\calO_{k,S}}$ of group schemes over $\Spec(\calO_{k,S})$, where $\Pic_{\calX/\calO_{k,S}}$ is the (smooth and proper) relative Picard scheme whose existence follows from the fact that $\calX$ is smooth and proper. Enlarge $S$ to ensure that
	\begin{enumerate}
		\item[(i)]\label{ietsha1} $\Sha^1_{S}(k,G/G_0) = \Sha^1(k,G/G_0)$,
		\item[(ii)]\label{three} $S$ includes all primes of residue cardinality up to the bound $B$.
	\end{enumerate}
	To see that (i) is possible, note that as $\Sha^1_S(k,G/G_0) - \Sha^1(k,G/G_0)$ is finite there is a finite set of primes $S'$ containing $S$ such that
	\[
		\Sha^1(k,G/G_0) = \ker\left( \Sha^1_S(k,G/G_0) \to \prod_{v \in S'} \HH^1(k_v,G/G_0)\right)\,.
	\]
	This kernel is $\Sha^1_{S'}(k,G/G_0)$.
	
	As noted in Remark~\ref{rmk:m}, the hypothesis of Lemma~\ref{lem:existenceLambda} is satisfied for $S$. Since $\Bcyr_{\lambda,S}(X) \subset \Br_{\lambda,S}(X)$ we may assume that there is a torsor $Y \to X$ of type $\lambda$ unramified outside $S$ (otherwise both sets in the statement of the theorem are empty by Lemma~\ref{lem:existenceLambda}). It will therefore suffice to verify that the second hypothesis of Lemma~\ref{lem:Sassumptions} is satisfied. 
	
	So, suppose $Y \to X$ is an unramified outside $S$ torsor of type $\lambda$ for which $Y(k_v) \ne \emptyset$ for $v \in S$. The scheme of connected components $\pi_0(Y)$ is a torsor under $G/G_0$ and the assumptions imply that $\pi_0(Y) \in \Sha^1_S(k,G/G_0) = \Sha^1(k,G/G_0)$. Thus $\pi_0(Y)$ is everywhere locally soluble. In particular, for any $v \not\in S$, $Y_v$ has a geometrically irreducible component defined over $k_v$ which is an $X_v$-torsor under $G_0$. This spreads out to a smooth scheme $\calY_0 \to \Spec(\calO_v)$ whose generic fiber is geometrically isomorphic to $\Ybar_0$. By~\eqref{three} and the discussion in the first paragraph we conclude that $Y(k_v) \ne \emptyset$. Thus, the second hypothesis of Lemma~\ref{lem:Sassumptions} is satisfied.
\end{proof}

\section{The Brauer-Manin obstruction for hyperelliptic curves}\label{sec:4}

Let $k$ be a field of characteristic $0$. Given a hyperelliptic curve $C/k$ with Jacobian $J$, consider the type map $\lambda_0 : J[2] = \Pic(\Cbar)[2] \subset \Pic(\Cbar)$. The main result of \cite{CreutzViray} is to give explicit representatives (as corestrictions of quaternion algebras over the function field $k(C)$) for the elements in a subgroup $\Br_\Upsilon(C)$ of $\Br_{\lambda_0}(C)$ (See \cite[Proposition 5.1]{CreutzViray}). In this section we show that, when $k$ is a number field, this subgroup captures the obstruction to the Hasse principle coming from $\Br(C)[2]$. We then use the results of the previous section to give a computable description of this obstruction.

\subsection{The group $\Br_\Upsilon(C)$}
To begin let us give a new definition of $\Br_{\Upsilon}(C)$ as a subgroup of the form $\Br_\lambda(C)$ for appropriate $\lambda$. The quotient by the hyperelliptic involution defines a map $\rho:C \to \PP^1$ which we may assume is not ramified over $\infty$. Let $\frak{m} = \rho^*\infty \in \Div(C)$. The $2$-torsion of the generalized Jacobian $J_\frak{m} = \Jac(C_\frak{m})$ sits in an exact sequence
\begin{equation}\label{Jm2}
	1 \to \mu_2 \to J_{\frak{m}}[2] \to J[2] \to 0\,.
\end{equation}
Let $\lambda : J_\frak{m}[2] \to \Pic(\Cbar)$ be the composition of $\lambda_0$ with the surjective map $J_{\frak{m}}[2] \to J[2]$. Now define $\Br_\Upsilon(C) = \Br_\lambda(C)$. When $k$ is a number field and $S \subset \Omega_k$ contains all archimedean primes we define $\Br_{\Upsilon,S}(C) = \Br_{\lambda,S}(C)$. 

We now recall the definition in \cite{CreutzViray} and observe that it is equivalent (the group there is denoted $\Br_2^\Upsilon(C)$). The long exact sequence of Galois cohomology groups from~\eqref{Jm2} gives an exact sequence 
\[
	\HH^1(k,J_{\frak{m}}[2]) \to \HH^1(k,J[2]) \stackrel{\Upsilon}\to \Br(k)\,.
\]
In \cite{CreutzViray} the group is defined in as $r^{-1}\left((\lambda_0)_*(\ker(\Upsilon))\right)$, where $r : \Br_1(C)/\Br_0(C) \simeq \HH^1(k,\Pic(\Cbar))$ is the canonical map. By the exact sequence above, this is equal to $\Br_\lambda(C) = r^{-1}(\lambda_*(\HH^1(k,J_\frak{m}[2])))$.

\begin{Theorem}\label{thm:BrUps}
	Let $C$ be a hyperelliptic curve over a number field $k$. Then
	\[
		C(\A_k)^{\Br(C)[2]} = C(\A_k)^{\Br_{\lambda_0}(C)} = C(\A_k)^{\Br_{\Upsilon}(C)}\,.
	\]
\end{Theorem}

\begin{proof}
The first equality is a special case of \cite[Corollary 6.2]{CreutzViray}. We must prove the second. As noted above $\Br_{\Upsilon}(C) = \Br_{\lambda}(C)$, where $\lambda : J_{\frak{m}}[2] \to J[2] \subset \Pic(\Cbar)$ is the composition of the quotient map in~\eqref{Jm2} with the inclusion $\lambda_0 : J[2] \subset \Pic(\Cbar)$.

A torsor $\pi:Y \to C$ of type $\lambda$ is a $C$-torsor under the finite group scheme $(J_\frak{m}[2])^\vee$, which is isomorphic to $\calJ[2] := \left(\Pic(\Cbar)/\langle \frak{m}\rangle\right)[2]$, see \cite[Proposition 2.8]{GeneralizedJacs}. The sequence~\eqref{Jm2} and its dual induce the vertical maps in the following commutative diagram.
\[
	\xymatrix{
	\HH^1(C,J[2]) \ar[r]^-{\textup{type}}\ar[d]& \Hom(J[2],\Pic(\Cbar)) \ar[d] \\
	\HH^1(C,\calJ[2]) \ar[r]^-{\textup{type}} \ar[d] &\Hom(J_{\frak{m}}[2],\Pic(\Cbar)) \ar[d] & \lambda \ar@{|->}[d] \\
	\HH^1(C,\Z/2\Z)) \ar[r]^-{\textup{type}} & \Hom(\mu_2,\Pic(\Cbar)) & 0 
	}
\]
(In the top left term we have used $J[2]^\vee \simeq J[2]$.) A $C$-torsor under $\Z/2\Z$ whose type is the $0$ map is a union of two copies of $C$ permuted by Galois and may be identified with a class in $\HH^1(k,\Z/2\Z)$. The image of a torsor $Y \to X$ of type $\lambda$ under the map $\HH^1(C,\calJ[2]) \to \HH^1(C,\Z/2\Z)$ is the scheme of connected components of $Y$, which geometrically is a pair of points. A connected torsor $(Y,\pi)$ of type $\lambda$ must have $Y(\A_k) = \emptyset$, since it is geometrically disconnected. Exactness of the first column in the diagram shows that the disconnected torsors of type $\lambda$ are precisely those that are the image of a torsor of type $\lambda_0$.

A torsor $\pi_0:Y_0 \to C$ of type $\lambda_0$ can be composed with the hyperelliptic involution $\iota : C \to C$ to obtain another torsor $\iota \circ \pi_0 : Y_0 \to C$ of type $\lambda_0$. The image of $(Y_0,\pi_0)$ under the map $\HH^1(C,J[2]) \to \HH^1(C,\calJ[2])$ is the disjoint union of $(Y_0,\pi_0)$ and $(Y_0,\iota\circ \pi_0)$. From this it follows that
\[
	\bigcup_{\substack{(Y_0,\pi_0) \in \HH^1(C,J[2])\,, \\ \textup{type}(Y_0,\pi_0)= \lambda_0}} \pi_0(Y_0(\A_k)) = \bigcup_{\substack{(Y,\pi) \in \HH^1(C,\calJ[2])\,, \\ \textup{type}(Y,\pi)= \lambda}} \pi(Y(\A_k))\,.
\]
By descent theory \cite[Theorem 6.1.2]{TorsorsAndRationalPoints} the term on the left is equal to $C(\A_k)^{\Br_{\lambda_0}(X)}$ and the term on the right is equal to $C(\A_k)^{\Br_{\lambda}(X)}=C(\A_k)^{\Br_{\Upsilon}(X)}$.
\end{proof}

\begin{Remark}\label{rmk:fakeSel}
	The restriction of the map $\HH^1(C,J[2]) \to \HH^1(C,\calJ[2])$ to the subset $\Sel^2(C/k)$ of locally soluble torsors is essentially the map $\Sel^2(C/k) \to \Sel^2_\textup{fake}(C/k)$ in \cite[Section 3]{BruinStoll}. The proof of the theorem relies on the fact that this map is surjective. This follows from the fact that the underlying curves of the torsors $(Y_0,\pi_0)$ and $(Y_0,\iota\circ\pi_0)$ are isomorphic.
\end{Remark}

\subsection{The $\mu$-map}\label{sec:mumap}
	Suppose that $C$ is defined by $y^2 = f(x)$ with $f(x) \in k[x]$ separable of even degree. We denote the leading coefficient of $f(x)$ by $c$. Define $L = k[x]/\langle f(x) \rangle$ and let $\theta \in L$ denote the image of $x$. Consider the $\mu$ map defined in \cite[Section 2]{BruinStoll}
	\[
		\mu : C(k) \to L^\times/k^\times L^{\times 2}\,,
	\] 
	which for a point $P = (a,b)$ with $b \ne 0$ is defined as $\mu(P) = (a-\theta)k^\times L^{\times 2}$. Let us further define $\frak{L}_1 = \ker\left( N_{L/k} : L^\times/k^\times L^{\times 2} \to k^{\times}/k^{\times 2}\right)$ and let $\frakL_c$ denote the coset of $\frakL_1$ of elements whose norm lies in $ck^{\times 2}$ (This set is denoted $H_k$ in \cite{BruinStoll}). Because $f(x) = cN_{L/k}(x-\theta)$, the image of $\mu$ is contained in $\frak{L}_c$.

\subsection{Torsors of the form $Y_\delta$}
As shown in \cite[Section 2]{BruinStoll} corresponding to any $\delta \in \frak{L}_c$ is a pair of two-coverings of $C$ over $k$ whose union gives a torsor $Y_\delta \to C$ of type $\lambda$. (The construction gives a pair of two-coverings because the morphism to $C$ can be composed with the hyperelliptic involution - see the bottom of page 2351 of \cite{BruinStoll}.)

\subsection{Cohomological interpretation}
	
	The group scheme $\calJ[2]$ sits in a short exact sequence $ 1 \to \mu_2 \to \Res_{L/k}(\mu_2) \to \calJ[2] \to 0$ identifying it with $\Res_{L/k}(\mu_2)/\mu_2$, where $\Res_{L/k}$ denotes the restriction of scalars functor taking $L$-schemes to $k$-schemes. The corresponding long exact sequence of Galois cohomology groups together with Hilbert's theorem 90 yields the exact sequence
\[
	k^\times/k^{\times 2} \to \left[ \HH^1(k,\Res_{L/k}(\mu_2)) = L^\times/L^{\times 2} \right]  \to \left[\HH^1(k,\Res_{L/k}(\mu_2)/\mu_2) = \HH^1(k,\calJ[2])\right]\,.
\]
 This fits together with the cohomology sequence of~\eqref{Jm2} to form the following commutative diagram with exact rows and columns (which is essentially \cite[(2.11)]{GeneralizedJacs} or \cite[(12)]{PoonenSchaefer} specialised to the present context).
		\begin{equation}\label{eq:maindiagram}
			\xymatrix{
				& k^\times/k^{\times 2}  \ar@{=}[r] \ar[d] &  k^\times/k^{\times 2} \ar[d] \\
				\Z/2 \ar[r]^-{d'}\ar@{=}[d]& \HH^1(k,J_{\frak{m}}[2]) \ar[r]\ar[d] & L^\times/L^{\times 2} \ar[r]^-{N_{L/k}} \ar[d] & k^{\times}/k^{\times 2} \ar@{=}[d] \\
				\Z/2 \ar[r]^-d &\HH^1(k,J[2]) \ar[d]^-\Upsilon \ar[r]& \HH^1(k,\calJ[2]) \ar[d]^-{\Upsilon'} \ar[r]^-N & k^{\times}/k^{\times 2}\\
				& \Br(k) \ar@{=}[r] & \Br(k)
			}
		\end{equation}
In particular, there is an injective map $L^\times/k^{\times}L^{\times 2} \hookrightarrow \HH^1(k,\calJ[2])$.

\subsection{Pairings}

The dualities $J_\frak{m}[2] = \calJ[2]^\vee$ and $\Res_{L/k}(\mu_2) = \Res_{L/k}(\Z/2)^\vee \simeq \Res_{L/k}(\mu_2)^\vee$ induce cup product pairings in Galois cohomology. Together with the maps in~\eqref{eq:maindiagram} and the identification $\Br(k) = \HH^2(k,\G_m)$ we obtain a commutative diagram of pairings:
		\begin{equation}\label{pairings}
	\begin{array}[c]{cccccc}
			\HH^1(k,J_\frak{m}[2]) & \times & \HH^1(k,\calJ[2])  & \to & \Br(k) \\
			\rotatebox{90}{$\twoheadleftarrow$} && \rotatebox{90}{$\hookrightarrow$} && \rotatebox{90}{$=$}\\
			(L^\times / L^{\times 2})_{N = 1}& \times & L^\times/k^\times L^{\times 2} & \to & \Br(k)\\
			\rotatebox{90}{$\hookleftarrow$} && \rotatebox{90}{$\twoheadrightarrow$} && \rotatebox{90}{$=$}\\
			L^\times / L^{\times 2}& \times & L^\times/L^{\times 2} & \to & \Br(k)
	\end{array}
	\end{equation}
	Here the subscript $N=1$ is used to indicate the kernel of the norm map. The bottom pairing is given explicitly by $\langle \ell, \ell' \rangle = \Cor_{L/k}(\ell,\ell')_L\,,$ where $(\ell,\ell')_L$ denotes the quaternion algebra over $L$ determined by $\ell$ and $\ell'$. That this induces a well defined pairing in the middle row follows from the fact that $\Cor_{L/k}(\ell,a)_L = (N_{L/k}(\ell),a)_k\in \Br(k)$, for any $a \in k^\times$ and $\ell \in L^\times$ (see \cite[Proposition 3.4.10(3)]{GS-csa}).
	
	\begin{Lemma}\label{lem:evaldelta}
		Let $\delta \in \frak{L}_c$ and $Y_\delta \to C$ be the corresponding torsor of type $\lambda$. The image of the evaluation map $C(k) \to \HH^1(k,\calJ[2])$ given by $Y_\delta \to C$ is $\delta^{-1}\mu(C(K)) \subset L^\times/K^{\times}L^{\times 2} \subset \HH^1(k,\calJ[2])$.
	\end{Lemma}
	
	\begin{proof}
		Evaluating the torsor $Y_\delta \to C$ at $P \in C(k)$ gives the class $\tau \in \HH^1(k,\calJ[2])$ of the fiber of $Y_\delta$ above $P$. Alternatively, $\tau$ is uniquely determined by the property that the point $P$ lifts to the twist of $Y_\delta$ by $\tau$. We will show that $\tau = \mu(P)/\delta$. By construction, a point $P \in C(k)$ lifts to $Y_\delta$ if and only if $\mu(P) = \delta$. Moreover, if $\delta' \in \frak{L}_1 \subset \HH^1(k,\calJ[2])$, then the twist of $Y_\delta$ by $\delta'$ is $Y_{\delta\delta'} \to C$. It follows that $P$ lifts to the twist of $Y_\delta \to C$ by $\mu(P)/\delta$ as required.
	\end{proof}

\subsection{Computable description of $C(\A_k)^{\Br(C)[2]}$}

\begin{Theorem}\label{thm:S1}
	Let $C : y^2 = f(x)$ be a locally soluble hyperelliptic curve over a number field $k$ of genus $g$ with $f(x) \in \calO_k[x]$. Suppose that $S \subset \Omega_k$ contains
	\begin{itemize}
	\item all archimedean primes,
	\item all primes above $2$,
	\item all primes that divide the leading coefficient of $f(x)$,
	\item all primes $v$ such that $\val_v(\disc(f(x))) \ge 2$,
\end{itemize}
	and that $\Sha^1_S(k,\mu_2) = 0$. Then
	\begin{enumerate}
		\item\label{it1} An adelic point $(P_v) \in C(\A_k)$ lies in $C(\A_k)^{\Br_{\Upsilon,S}(X)}$ if and only if there exists a two-covering $\pi:Y \to C$ unramified outside $S$ such that for all $v \in S$, $\pi(Y(k_v))$ contains either $P_v$ or $\iota(P_v)$. In particular, $C(\A_k)^{\Br_{\Upsilon,S}(X)} \ne \emptyset$ if and only if there exists a two-covering of $C$ that is unramified outside $S$ and soluble at all primes of $S$.
		\item\label{it2} If $S$ contains all primes whose residue cardinality $q$ satisfies
\[
	\sqrt{q} + \frac{1}{\sqrt{q}} \le 2(2^{2g}(g-1)+1)\,,
\]
	  then
	  \[
	  	C(\A_k)^{\Br_{\Upsilon,S}(C)} = \emptyset \quad \Leftrightarrow \quad C(\A_k)^{\Br(C)[2]} = \emptyset \quad \Leftrightarrow \quad \Sel^2(C/k) = \emptyset\,.
	  \]
	  \end{enumerate}
\end{Theorem}

\begin{proof}
		To prove~\eqref{it1}, first suppose there is a two-covering $\pi:Y \to C$ with the stated property. Composing with the hyperelliptic involution gives another such torsor and the union of these is a torsor of type $\lambda : J_\frak{m}[2] \to \Pic(\Xbar)$ that is unramified outside $S$ and which contains a lift of $P_v$ for every $v \in S$. By Lemma~\ref{lem:BrlambdaS} we conclude that $(P_v) \in C(\A_k)^{\Br_{\Upsilon,S}(C)}$. 
		
		For the converse, suppose $(P_v) \in C(\A_k)^{\Br_{\Upsilon,S}(C)} \ne \emptyset$. For $v \not\in S$ it follows from \cite[Lemma 4.3]{BruinStoll} that there exists an unramified $C_{k_v}$-torsor of type $\lambda$. Indeed, the lemma shows that the image of $\mu : C(k_v) \to L_v^\times/k_v^\times L_v^{\times 2} \subset \HH^1(k_v,\calJ[2])$ lies in the unramified subgroup. Given $\delta_v$ in the image, the corresponding torsor $Y_{\delta_v} \to C_v$ is unramified by Lemma~\ref{lem:evaldelta}. By Lemma~\ref{lem:existenceLambda} we thus conclude that there exists a $C$-torsor of type $\lambda$ that is unramified outside $S$ (note that $(P_v) \in C(\A_k)^{\Br_{\Upsilon,S}(C)} \subset C(\A_k)^{\Bcyr_{\lambda,S}(C)}$). This shows that the hypothesis of Lemma~\ref{lem:BrlambdaS} is satisfied and so we may conclude that there is a torsor of type $\lambda$ unramifed outside $S$ which contains a lift of $P_v$ for each $v \in S$. The scheme of connected components of this torsor represents an element of $\Sha^1_S(k,\mu_2)$. By assumption $\Sha^1_S(k,\mu_2) = 0$, so its geometric components are defined over $k$. These components are two-coverings of $C$ which differ by the hyperelliptic involution. If $P_v$ lifts to one of them, then $\iota(P_v)$ lifts to the other.
		
		Now let us prove~\eqref{it2}. It is well known that $\Sel^2(C) = \emptyset \Leftrightarrow C(\A_k)^{\Br_{\lambda_0}(C)}= \emptyset$ (see \cite[Theorem 6.1.2]{TorsorsAndRationalPoints}) and $C(\A_k)^{\Br_{\lambda_0}(C)} = C(\A_k)^{\Br(C)[2]} = C(\A_k)^{\Br_{\Upsilon}(C)}$ by Theorem~\ref{thm:BrUps}. If $C(\A_k)^{\Br_{\Upsilon,S}(C)}= \emptyset$, then all three sets in the statement are empty. So suppose $C(\A_k)^{\Br_{\Upsilon,S}(C)} \ne \emptyset$. By the preceding discussion it will suffice to show that $C(\A_k)^{\Br_{\Upsilon}(C)} \ne \emptyset$. For this we use Lemma~\ref{lem:Sassumptions}. Above we have shown that there is a torsor $Y \to C$ of type $\lambda : J_{\frak{m}}[2] \to \Pic(\Xbar)$ unramified outside $S$, verifying the first assumption in Lemma~\ref{lem:Sassumptions}. To verify the second suppose that $Y \to C$ is a torsor of type $\lambda$ unramified outside $S$ and soluble on $S$. Let $v \not\in S$. We must show that $Y(k_v) \ne \emptyset$. 
		
		For any $\delta \in \mu(C(k_v))$, the proof of \cite[Lemma 4.3]{BruinStoll} shows that the corresponding torsor $Y_{\delta} \to C_v$ has $Y_\delta(k_v) \ne \emptyset$ (This is where the bound on the residue cardinality of $v$ in the hypothesis is used). It will therefore suffice to show that our $Y$ above is isomorphic over $k_v$ to one of the form $Y_\delta$. Note that $Y\to C$ is a twist of $Y_\delta \to C$ by an element $\tau \in \HH^1(k_v,\calJ[2])$ in the image of $\HH^1(k_v,J[2])$, since both are disconnected $C$-torsors under $\calJ[2] = J_\frak{m}[2]^\vee$ defined over $k_v$. The computation in \cite[Lemma 4.3]{BruinStoll} shows that the image $I$ of the evaluation map $Y_\delta : C(k_v) \to \HH^1(k_v,\calJ[2])$ is equal to the set of elements in $\ker(\Upsilon') \cap \ker(N) = \frakL_{v,1}$ that are unramified, where $N$ and $\Upsilon'$ are the maps in~\eqref{eq:maindiagram} (again, using that the residue cardinality is sufficiently large). The image of the evaluation map corresponding to the twist $Y$ is the coset $\tau I$. Since $Y \to C$ is unramified at $v$, $\tau I$ is contained in the unramified subgroup of $\HH^1(k_v,\calJ[2])$, and so $\tau$ itself is unramified. By \cite[Lemma 2.10]{GeneralizedJacs} $\Upsilon'(\tau) = \tau \cup d'(1)$. Since $f(x)$ has discriminant of valuation $\le 1$, $f(x)$ has a root over $k_v^{\unr}$ (see the proof of~\cite[Lemma 4.3]{BruinStoll}). This implies that $d'(1)$ is unramified and, hence, that $\Upsilon'(\tau) = \tau \cup d'(1) = 0$ because the unramified subgroups are orthogonal. From the diagram~\eqref{eq:maindiagram} we conclude that $\tau$ is the image of some $\delta' \in L_v^\times/L_v^{\times 2}$ which is unramified and of square norm. Then $\delta\delta'$ is unramified with norm in $ck_v^{\times 2}$ and so $Y = (Y_\delta)^\tau = Y_{\delta\delta'}$ is of the required form.
\end{proof}

\begin{Remark}\label{rmk:notsobadv}
	Suppose $S$ is any set of primes as in the first part of Theorem~\ref{thm:S1}. Let $v \not\in S$, let $\delta \in L_v^\times$ be an unramified element of norm $c$ times a square and let $Y_\delta$ be the corresponding torsor of type $\lambda$. In the course of the proof we have shown the following: if the evaluation map corresponding to $Y_\delta$ surjects onto the set of unramified elements in the subgroup $\ker(\Upsilon') \cap \ker(N) \subset \HH^1(k_v,\calJ[2])$, then any torsor $Y \to C$ of type $\lambda$ which is unramified outside $S$ and soluble on $S$ is also soluble at $v$. For such $v$ it follows from Lemma~\ref{lem:BrlambdaS} that
	\[
		C(\A_k)^{\Br_{\Upsilon,S}(C)} \ne \emptyset \quad \Leftrightarrow \quad  C(\A_k)^{\Br_{\Upsilon,S\cup\{v\}}(C)} \ne \emptyset\,.
	\]
	The condition is satisfied at $v$ when the reside cardinality is larger than the bound in the statement of Theorem~\ref{thm:S1}, but it is typically also satisfied for many primes outside $S$ below the bound as well. This observation is quite useful in practice as it allows us to identify (and safely omit) such primes from the computation.
\end{Remark}

\begin{Remark}\label{Sha1mu2}
For any number field $k$, $\Sha^1(k,\mu_2) = 0$ by the Grunwald-Wang theorem. It follows (cf. the proof of Theorem~\ref{thm:BrlambdaS}) that one can always enlarge $S$ if necessary to ensure that $\Sha^1_S(k,\mu_2) = 0$ as in the hypothesis of Theorem~\ref{thm:S1}. In the case $k = \Q$, we have $\Sha^1_S(\Q,\mu_2) = 0$ for any $S$ containing the archimedean prime, but this is not true in general. For example, if $k = \Q(\sqrt{34})$, $v$ is the prime above $2$ and $S = \{ v, \infty_1,\infty_2 \}$, then $k_{v} = \Q_2(\sqrt{2})$ so the class of $2$ is a nontrivial element in $\Sha^1_{S}(k,\mu_2) \subset \HH^1(k,\mu_2) = k^{\times}/k^{\times 2}$. Note however, that $3$ splits in $\calO_k$ and $2 \notin \Q_3^{\times 2}$. If $S$ is enlarged to include a prime above $3$, then $\Sha^1_{S}(k,\mu_2) = 0$.
\end{Remark}

For the remainder of the section we suppose $C/k$ is the hyperelliptic curve defined by the affine equation $y^2 = f(x)$ and maintain the notation introduced above.

\subsection{Explicit representatives for $\Br_\Upsilon(C)$} 
By \cite[Proposition 5.1]{CreutzViray} the homomorphism
\begin{equation}\label{eq:gamma}
	L^\times \ni \ell \mapsto \calA_\ell := \Cor_{L/k}(\ell,x-\theta) \in \Br(k(C))
\end{equation}
induces a surjective map
\[
	\frakL_1 \to \Br_\Upsilon(C)/\Br_0(C)\,.
\]

\subsection{The unramified outside $S$ subgroup}

\begin{Definition}
	Suppose that $k$ is number field. A class in $L^\times/L^{\times 2}$ is \defi{unramified at $v \in \Omega_k$} if its image under the isomorphism $L^\times/L^{\times 2} = \HH^1(k,\Res_{L/k}(\mu_2))$ is unramified at $v$. A class in $L^\times/k^\times L^{\times 2}$ is \defi{unramified at $v \in \Omega_k$} if its image under the injecitve map $L^\times/k^\times L^{\times 2} \to \HH^1(k,\Res_{L/k}(\mu_2)/\mu_2)=\HH^1(k,\calJ[2])$ is unramified at $v$. For a subset $S \subset \Omega_k$ we use $(L^\times/L^{\times 2})_S$, $\frak{L}_{1,S}$ and $\frak{L}_{c,S}$ to denote the subsets (of $L^\times/L^{\times 2}$, $\frak{L}_1$ and $\frak{L}_c$) of elements that are unramified at all primes outside of $S$.
\end{Definition}

\begin{Lemma}\label{lem:L1SS}
	Suppose $k$ is a number field. For any $S \subset \Omega_k$ satisfying the hypothesis of Theorem~\ref{thm:S1}, the map in~\eqref{eq:gamma} induces a surjective map
	\[
		\frakL_{1,S} \to \Br_{\Upsilon,S}(C)/\Br_0(C)\,.
	\]
\end{Lemma}

\begin{proof}
	Let $v \not\in S$. To begin, let us observe that an element of $\HH^1(k_v,J_\frak{m}[2])$ is unramified if and only if its image in $L_v^{\times}/L_v^{\times 2}$ is unramified. To see this we use exactness of the top row of~\eqref{eq:maindiagram} and the fact, noted in the proof of Theorem~\ref{thm:S1}, that $d'(1) \in \HH^1(k_v,J_\frak{m}[2])$ is unramified.
	
	Now suppose $\ell' \in L^\times$ represents a class in $\frak{L}_{1,S}$. Then there exists $a \in k^\times$ such that the class of $\ell = a\ell'$ lies in $(L^{\times}/L^{\times 2})_S$. Note that $\calA_\ell = \calA_{\ell'} + \calA_a =  \calA_{\ell'} + \Cor_{L/k}(a,x-\theta) = \calA_{\ell'} + (a,N_{L/k}(x-\theta)) =\calA_{\ell'} + (a,c) \in \calA_{\ell'}+\Br_0(C)$ (here $c \in k^\times$ is the leading coefficient of $f(x) = cN_{L/k}(x-\theta)$). Exactness of~\eqref{eq:maindiagram} shows that there exists $\alpha \in \HH^1(k,J_\frak{m}[2])$ mapping to the class of $\ell$. By the observation in the first paragraph of the proof, $\alpha \in \HH^1_S(k,J_\frak{m}[2])$. The proof of~\cite[Proposition 5.1]{CreutzViray} shows that $r(\calA_\ell) = \lambda_*(\alpha)$, so $\calA_\ell \in \Br_{\Upsilon,S}(C)$.
	
	To show surjectivity, suppose $\beta \in \Br_{\Upsilon,S}(C)$. Then there exists $\alpha \in \HH^1_{S}(k,J_\frak{m}[2])$ such that $r(\beta) = \lambda_*(\alpha)$ in $\HH^1(k,\Pic(\Cbar))$. If $\ell$ represents the image of $\alpha$ in $L^\times/L^{\times 2}$, then $r(\calA_\ell) = \lambda_*(\alpha)$ so $\beta$ and $\calA_\ell$ determine the same class modulo $\Br_0(C)$. To prove surjectivity we must show that the image of $\ell$ in $L^\times/L^{\times 2}$ is unramified at $v$. 
	
	By assumption $\alpha$ annihilates the unramified subgroup of $\HH^1(k_v,\calJ[2])$. From~\eqref{pairings} it follows that the class of $\ell$ in $L_v^\times/L_v^{\times 2}$ annihilates the unramified subgroup of $L_v^{\times}/L_v^{\times 2}$. For an odd prime such as $v \not\in S$, the unramified subgroup of $L_v^{\times}/L_v^{\times 2}$ is its own exact annihator (even when some simple factors $L_i$ of $L$ are ramified over $k$). So the class of $\ell$ is unramified as required.
\end{proof}

\subsection{Evaluation of $\calA_\ell$}
Given $\calA \in \Br(C)$ and a point $P : \Spec(k) \to C \in C(k)$, the evaluation of $\calA$ at $P$, denoted by $\calA(P)$, is defined as the pullback $P^*\calA \in \Br(k)$.

\begin{Lemma}\label{lem:AP}
	Let $\ell \in (L^\times/L^{\times 2})_{N=1}$ and $P \in C(k)$. Then $\calA_\ell(P) = \Cor_{L/k}(\ell,\mu(P)) \in \Br(k)$.
\end{Lemma}

\begin{proof}
	This follows from the definition of $\calA_\ell$ and the definition of $\mu$ given in Section~\ref{sec:mumap}.
\end{proof}

\begin{Lemma}\label{lem:HS}
	Suppose $k$ is a local field of characteristic $0$, $\ell \in (L^\times/L^{\times 2})_{N=1}$ and $P \in C(k)$. Write $L = \prod_{i = 1}^n L_i$ as a product of finite extensions of $k$ and let $\ell = (\ell_i)$ and $\mu(P) = (m_i)$ be the images under this decomposition. Then 
	$$\inv_k(\calA_\ell(P)) = \sum_{i = 1}^n (\ell_i,m_i)_{L_i},\,$$
	where $(\ell,m_i)_{L_i} \in \{ 0 , 1/2 \} \subset \Q/\Z$ is the (additive) Hilbert symbol on the local field $L_i$.
\end{Lemma}

\begin{proof}
By Lemma~\ref{lem:AP}, 
\[\inv_k(\calA_\ell(P)) = \inv_k(\Cor_{L/k}(\ell,\mu(P))) = \sum_{i=1}^n\inv_k( \Cor_{L_i/k}(\ell_i,m_i))\,,\]
and $\inv_k( \Cor_{L_i/k}(\ell_i,m_i)) = \inv_{L_i}(\ell,m_i) = (\ell,m_i)_{L_i}$, since $\inv_k\circ \Cor_{L_i/k} = \inv_{L_i}$ for the extension of local fields $L_i/k$.
\end{proof}

\begin{Lemma}\label{lem:Aelladelic}
	Suppose $k$ is a number field and $S \subset \Omega_k$ is a finite set of primes satisfying the hypothesis of Theorem~\ref{thm:S1}. Suppose $\ell \in (L^\times/L^{\times 2})_S$ has square norm and $(P_v) \in C(\A_k)$. Then
	\[
		\sum_{v \in \Omega_k}\inv_v(\calA_\ell(P_v)) = \sum_{v \in S}(\ell,\mu_v(P_v))_v\,.
	\] 
	where $\mu_v : C(k_v) \to L_v^{\times}/k_v^\times L_v^{\times 2}$ is the map from Section~\ref{sec:mumap} and $(\ell,\mu_v(P_v))_v$ is the sum of the Hilbert symbols on the simple factors of $L_v$. 
\end{Lemma}

\begin{proof}
Any $\ell \in (L^\times/L^{\times 2})_S$ is unramified at all $v\notin S$, and by \cite[Lemma 4.3]{BruinStoll}, $\mu_v(P_v)$ is unramified at all $v\notin S$ (as per the discussion in proof of \ref{thm:S1}). Thus,  $\inv_v(\calA_\ell(P_v))=0$ at all $v\notin S$. So 
\[\sum_{v \in \Omega_k}\inv_v(\calA_\ell(P_v))=\sum_{v \in S}\inv_v(\calA_\ell(P_v))\,.\]
The result follows from Lemma \ref{lem:HS}.

\end{proof}

\section{The algorithm}\label{sec:algorithm}

Let $C : y^2 = f(x)$ be a locally soluble hyperelliptic curve over a number field $k$ of genus $g$ with $f(x) \in \calO_k[x]$. Let $S_{\min} \subset \Omega_k$ be a finite set of primes containing
	\begin{itemize}
	\item all archimedean primes,
	\item all primes above $2$,
	\item all primes that divide the leading coefficient of $f(x)$,
	\item all primes $v$ such that $\val_v(\disc(f(x))) \ge 2$, and
	\item enough primes so that $\Sha^1_S(k,\mu_2) = 0$.
\end{itemize}

\begin{Algorithm}\label{alg1} Let $C/k$ be as at the beginning of this section. Let $\ell_1,\dots,\ell_n \in L^{\times}$ be elements of square norm.
	\begin{enumerate}
	\item\label{step0} Let $S \subset \Omega_k$ be the set of primes obtained by adding to $S_{\min}$ all primes where some $\ell_i$ is ramified. 
	\item\label{step2} For $i = 1,\dots,n$ compute the $\F_2$-linear map
	\[
		\phi_i : \prod_{v \in S} L_v^\times/k_v^{\times}L_v^{\times 2} \ni (m_v) \mapsto \sum_{v \in S} \inv_v\langle \ell_i,m_v\rangle_v \in \frac{1}{2}\Z/\Z\,, 
	\]
	where $\langle \ell_i,m_v\rangle_v = \Cor_{L_v/k_v}(\ell_i,m_v) \in \Br(k_v)$ denotes the central pairing in~\eqref{pairings} (over the base field $k_v$). 
	\item\label{step3} For $v \in S$ compute the image $I_v$ of the map 
	\[
		\mu_v : C(k_v) \to L_v^\times/k_v^\times L_v^{\times 2}\,,
	\]
	defined in Section~\eqref{sec:mumap}.
	\item\label{step4} Compute and return the intersection of $\prod_{v \in S} I_v$ and $\cap_{i = 1}^n \ker(\phi_i)$.
	\end{enumerate}
\end{Algorithm}

\begin{Proposition}\label{prop:alg1}
	The output of Algorithm~\ref{alg1} is the set $\prod_{v \in S}\mu_v(\pi_v(C(\A_k)^{B}))$, where $B \subset \Br(C)$ is the subgroup generated by the algebras $\calA_{\ell_1},\dots,\calA_{\ell_n}$ corresponding to $\ell_1,\dots,\ell_n$ and $\pi_v : C(\A_k) = \prod_{v \in \Omega_k}C(k_v) \to C(k_v)$ is the canonical projection. In particular,
	\begin{enumerate}
		\item the output is the empty set if and only if $C(\A_k)^B = \emptyset$, and
		\item if $\ell_1,\dots,\ell_n$ represent a basis for $\frak{L}_{1,S}$, then the output is the empty set if and only if $C(\A_k)^{\Br_{\Upsilon,S}(C)}= \emptyset$.
	\end{enumerate}
\end{Proposition}

\begin{proof}
An element $(m_v)_{v \in S} \in \prod_{v\in S} L_v^\times/k_v^\times L_v^{\times 2}$ lies in $\prod_{v\in S} I_v$ if and only if there exists $(P_v) \in C(\A_k)$ such that $m_v = \mu_v(P_v)$ for $v \in S$. For such $(P_v)$, Lemmas~\ref{lem:HS} and~\ref{lem:Aelladelic} give
\[
	\sum_{v\in \Omega_k} \inv_v\calA_{\ell_i}(P_v) = \sum_{v\in S} \inv_v\Cor_{L_v/k_v}(\ell_i, \mu_v(P_v)) = \sum_{v\in S} \inv_v\langle \ell_i, \mu_v(P_v) \rangle = \phi_i((m_v))\,.
\]
So $P \in \cap_{i = 1}^nC(\A_k)^{\calA_{\ell_i}} = C(\A_k)^B$ if and only if $(m_v) \in \cap_{i = 1}^n\ker(\phi_i)$, in which case $(m_v) \in \prod_{v \in S}\mu_v(\pi_v(C(\A_k)^{B}))$.

Statement~\eqref{it1} follows immediately. 

For~\eqref{it2}, suppose $\{\ell_1,\dots,\ell_n\}$ represent a basis for $\frak{L}_{1,S}$. Then $B\subset \Br_{\Upsilon,S}(C)$ and by Lemma~\ref{lem:L1SS} we have that $B$ surjects onto $\Br_{\Upsilon,S}(C)/\Br_0(C)$. So $C(\A_k)^B = C(\A_k)^{\Br_{\Upsilon,S}(C)}$ and~\eqref{it2} follows.
\end{proof}

We now provide details of how to carry out the steps in the algorithm above.

\noindent{\bf Step~\eqref{step0}:} For an odd prime $v$ one can detect whether $\ell_i$ is ramified at $v$ by looking at the valuations of (the components of) $\ell_i$ at the primes above $v$, so this step is straightforward provided we can factor the numerator and denominator of the norm $N_{L/\Q}(\ell_i)$. In practice, we often choose $S$ and then find $\ell_1,\dots,\ell_n$ that are unramified outside $S$ to use as the input.

\noindent{\bf Step~\eqref{step2}:} As noted in Lemma~\ref{lem:HS}, the pairings $\inv_v\langle \ell_i,m_v\rangle_v$ can be computed using Hilbert symbols. We choose a basis for each of the finite $\F_2$ vector spaces $L_v^{\times}/k_v^{\times}L_v^{\times 2}$ and compute the pairings of these basis elements with $\ell_i$. This allows us to write down a matrix (column vector) representing the linear map $\phi_i$.

\noindent{\bf Step~\eqref{step3}:} An efficient algorithm for this is described in \cite[Section 4]{BruinStoll}.

\noindent{\bf Step~\eqref{step4}:} The set to be computed is the intersection of the subspace $\cap \ker(\phi_i)$ of the finite vector space $V := \prod_{v \in S} L_v^\times/k_v^{\times}L_v^{\times 2}$ over $\F_2$ with a finite subset $I := \prod_{v\in S}I_v \subset V$. While this is clearly computable, $\#I$ grows exponentially with $\#S$, so the naive approach of testing each element of $I$ for membership of $W$ will quickly become impractical. We describe a recursive method for computing this intersection which makes use of the fact that the $\phi_i$ are sums linear maps on the $L_v^\times/k_v^{\times}L_v^{\times 2}$ and $I$ is a product of subsets of these spaces.

Call a subset $X \subset \prod_{v \in S} L_v^\times/k_v^\times L_v^{\times 2}$ a \defi{subproduct} if it is a Cartesian product of subsets $X = \prod_{v \in S} X_v$ with $X_v \subset L_v^\times/k_v^{\times}L_v^{\times 2}$. For example, $I$ is a subproduct. Suppose $\ell \in L^\times$ is unramified outside $S$. For each $v \in S$, $\ell$ determines a partition $X_v = X_v^0 \cup X_v^1$ where $X_v^0 = \{ m_v \in X_v \;:\;  \inv_v\langle\ell,m_v\rangle = 0 \}$. It follows that $X^\ell := X \cap \ker(\phi_\ell)$ is the union the $2^{(\#S - 1)}$ (possibly empty) subproducts $X_{\bf a} := \prod_{v \in S}X^{a_v}_v$, where ${\bf a} = (a_v) \in (\F_2)^{S}$ is such that $\sum a_v = 0$. Now if $\ell' \in L^\times$ is another element unramified outside $S$, we see
\begin{equation}\label{eq:twoells}
	X^{\{\ell,\ell'\}} = X^\ell \cap X^{\ell'} = \bigcup_{{\bf a}} ( X_{\bf a} \cap X^{\ell'}) = \bigcup_{{\bf a}} (X_{\bf a})^{\ell'}\,,
\end{equation}
where the sets $(X_{\bf a})^{\ell'}$ are themselves unions of $2^{(\#S-1)}$ subproducts by the argument above, which applies because the $X_{\bf a}$ are subproducts. It follows that $I^{\{\ell_1,\dots,\ell_n\}} = I \cap \left(\cap_{i = 1}^n \ker(\phi_i)\right)$ can be written as a union of $2^{(\#S-1)^n}$ subproducts. These may be viewed as the leaves of a regular $(\#S-1)$-ary tree of height $n$ whose nodes at height $m$ correspond to the subproducts in $I^{\{\ell_1,\dots,\ell_m\}}$ (the root of the tree is $I$ and lies at height $0$). To determine which of the leaves correspond to nonempty subproducts we traverse each branch upward until we reach either an empty node (i.e. node whose subproduct is empty) in which case all nodes above this are empty, or find a nonempty node at level $n$ in which case the corresponding elements lie in $I^{\{\ell_1,\dots,\ell_n\}}$.

Let us discuss the efficiency of this approach. Since the dimension of $L_v^\times/k_v^{\times}L_v^{\times 2}$ is bounded by a constant $c$ independent of $v$, computing the nodes $X_{\bf a}$ comprising $X^\ell$ below a given node $X$ requires checking subspace membership for at most $\sum_{v\in S} \#X_v \le c \#S$ elements (this is linear in $\#S$, whereas the naive approach would require $\prod_{v \in S} \#X_v$ membership checks, which is exponential in $\#S$). In a worst case scenario where all interior nodes are nonempty (so we must traverse the entire tree) we would need to call the procedure $2^{{(\#S - 1)}^{n}}$ times. In practice though, the situation is much better. For given $\ell$, there are typically very few nodes $X_{\bf a} \subset X^\ell$ which are nonempty, meaning most branches are quite short. In part this is explained by the following observation: the sets $X^1_v$ are empty except possibly when $v$ lies in the set
\[
	S'(\ell) := \left\{ v \in S \;:\; \text{$\ell$ is ramified at $v$, $v \mid 2c$, $v$ is archimedean, or $\ord_v(\disc(f)) \ge 2$} \right\}\,,
\]
which is typically much smaller than $S$. This is because for primes $v \not\in S'(\ell)$ both the image of $\ell$ in $L_v^\times/L_v^{\times 2}$ and the local image $I_v$ are unramified (cf. the proof of Theorem~\ref{thm:S1}). 

\subsection{Choosing the input for Algorithm~\ref{alg1}}

Suppose $S \subset \Omega_k$ contains $S_{\min}$ and let $B \subset \Br(C)$ be the subgroup generated by $\calA_{\ell_1},\dots,\calA_{\ell_n}$, where $\ell_1,\ell_2,\ldots,\ell_n \in L^\times$ are elements of square norm unramified outside $S$. The algorithm allows us to determine if $C(\A_k)^{B}$ is empty. We have containments
\[
	C(k) \subset C(\A_k)^{\Br(C)[2]} \subset C(\A_k)^{\Br_{\Upsilon,S}(C)} \subset C(\A_k)^B \subset C(\A_k)\,.
\]
If $\{ \ell_1,\dots,\ell_n\}$ spans $\frak{L}_{1,S}$, then Lemma~\ref{lem:L1SS} shows that $C(\A_k)^{\Br_{\Upsilon,S}(C)} = C(\A_k)^B$. As described in \cite[Section 12]{PoonenSchaefer}, a basis for $\frak{L}_{1,S}$ can be computed from the $S$-class and $S$-unit groups of the rings of integers of the simple factors of $L$. In practice this may only be feasible under the assumption of GRH. The standard algorithms used to compute class and unit groups will produce a basis for a subspace of $\frak{L}_{1,S}$, which is equal to $\frak{L}_{1,S}$ under the assumption of GRH. Consequently, if $\ell_1,\dots,\ell_n$ is such a conditional basis, the equality $C(\A_k)^{\Br_{\Upsilon,S}(C)} = C(\A_k)^B$ is conditional on GRH, but the containment $C(k) \subset C(\A_k)^{\Br_{\Upsilon,S}(C)} \subset C(\A_k)^B$ is not. In particular, if we conclude from Algorithm~\ref{alg1} that $C(\A_k)^B = \emptyset$, then it follows unconditionally that $C(\A_k)^{\Br_{\Upsilon,S}(C)} = C(k) = \emptyset$.

By Theorem \ref{thm:S1}, the containment $C(\A_k)^{\Br(C)[2]} \subset C(\A_k)^{\Br_{\Upsilon,S}(C)}$ becomes an equality provided $S$ contains all primes of norm up to the bound in the theorem (which depends on the genus of $C$). Consequently, Algorithm~\ref{alg1} gives an algorithm to check if $C(\A)^{\Br C[2]}=\emptyset$. Including all primes up to that bound quickly becomes impractical (the bounds are $1158$, $66562$, $2365446$,  for $g = 2,3,4$). However, in practice we have found that examples with
\[
	C(\A_k)^{\Br_{\Upsilon,S_{\min}}(C)} \ne \emptyset \quad \text{ and } C(\A_k)^{\Br_{\Upsilon,S}(C)} = \emptyset
\]
are fairly uncommon. In such cases, one usually finds that this holds with $S$ including only primes of bad reduction and primes well below the bound in Theorem~\ref{thm:S1}. 
Among the samples described in the following section we found only one curve where including a prime of good reduction larger than $100$ had an impact; this genus $5$ curve is considered in Section~\ref{subsec:g5} below. This behavior is similar to that seen in the descent algorithm of \cite{BruinStoll}, as can be explained by Theorem~\ref{thm:S1}(1) (see also Remark~\ref{rmk:notsobadv}). The descent algorithm computes the finite set of two-coverings locally soluble outside $S_{\min}$ (conditionally on GRH) and then checks these for local solubility at the primes below the bound. Only the primes where one of these coverings is not soluble will be relevant for the Brauer-Manin obstruction.

For curves of high genus or with large coefficients computation of $\frak{L}_{1,S}$ may be prohibitive even assuming GRH. In this case we may still be able to compute some elements of $\frak{L}_{1,S}$ which may be enough to show $C(\A_k)^B = \emptyset$, allowing us to conclude that $C(k) = \emptyset$ (for an example see Proposition~\ref{prop:g50}). In fact, the standard algorithm for computing $\frak{L}_{1,S}$ proceeds by generating random elements until sufficiently many have been found (assuming GRH gives a better stopping criterion). Computing the obstructions coming from such randomly generated elements may lead to a more efficient Las Vegas style algorithm for computing Brauer-Manin obstructions. This idea will be pursued elsewhere.

\section{Examples and Data}\label{sec:Data}

We have implemented Algorithm~\ref{alg1} for hyperelliptic curves over $\Q$ in magma \cite{magma} and used this to produce the examples and data described in this section\footnote{This code is available via at:\\
\url{https://github.com/brendancreutz/Brauer-Manin_Obstructions_on_hyperelliptic_curves}\\
 or with the source files of the arXiv distribution of this paper, arXiv:2112.00230v2.}.

\subsection{A genus $5$ example}\label{subsec:g5}
Consider the genus $5$ curve $C/\Q$ defined by $y^2 = f(x)$, where
\[f(x) = -17x^{12} - 13x^{11} - 15x^{10} + 6x^9 - 19x^8 + 5x^7 - 19x^6 + 4x^5 - 2x^4 +19x^3 + 12x^2 + 13x - 6\,.\]
This curve is everywhere locally soluble. The discriminant of $f(x)$ factors as
\[
	\disc(f(x)) = 2^6\times5^2\times 29 \times 151 \times 54918937 \times 571571633 \times 8389309314807991\,,
\]
so we have $S_{\min} = \{ 2,5,17,\infty \}$. Under the assumption of GRH we can compute a basis for $\frak{L}_{1,S_{\min}}$ in a couple of seconds. Let $B \subset \Br_{\Upsilon,S_{\min}}(C)$ denote the subgroup spanned by this conditional basis. Using Algortihm~\ref{alg1} we compute the image of the set $C(\A_\Q)^B$ in $\prod_{v \in S_{\min}}L_v^\times/k_v^{\times}L_v^{\times 2}$ and find that it is not empty. This does not imply $C(\A_\Q)^{\Br(C)[2]} \ne \emptyset$ (even assuming GRH) because we have not considered elements in $\Br_{\Upsilon}(C)$ that ramify only at the primes below the bound in Theorem~\ref{thm:S1}(2), which in this case is $67141638$. To obtain a deeper obstruction we can repeat the computation, replacing $S_{\min}$ by $S_N := S_{\min} \cup \{ p \le N \;:\; p \text{ is prime} \}$ for a suitable bound $N$. In practice $N = 1000$ is entirely feasible and we found that $C(\A_\Q)^{\Br_{\Upsilon,S_{1000}}(C)} = \emptyset$, proving that $C$ is a counterexample to the Hasse principle. The entire computation took about 12 minutes on our server.

As we have found is typically the case, the obstruction can also be obtained using a much smaller set of primes $S$. The \verb+TwoCoverDescent+ in Magma described in \cite{BruinStoll} computes the sets
\[
	\Sel^2(C/\Q)_N := \{ \text{Two-covers that are locally soluble for all $p \le N$ and all $p > 67141638$} \}
\]
for increasing $N$. After about 20 minutes on our server it concluded that $\Sel^2(C/\Q)_{238} \ne \emptyset$, but $\Sel^2(C/\Q)_{239} = \emptyset$ (conditionally on GRH). In light of Theorem~\ref{thm:S1}, we therefore expect the prime $p = 239$ will be relevant for the Brauer-Manin obstruction as well. Running Algorithm~\ref{alg1} with a (GRH-conditional) basis for $\frak{L}_{1,S}$, with $S = \{2,5,17,239,\infty\}$ takes about 15 seconds and shows unconditionally that the sets $C(\Q),\, C(\A_\Q)^{\Br(C)[2]},\, C(\A_\Q)^{\Br_{\Upsilon,S}(C)}$ and $\Sel^2(C/\Q)_{239}$ are all empty. In this way, Algorithm~\ref{alg1} can be used to certify the otherwise conditional output of \verb+TwoCoverDescent+. That said, in the vast majority of cases where we have seen a Brauer-Manin obstruction it is already given by $\Br_{\Upsilon,S_{\min}}(C)$ and (for higher genus curves at least) it seems computation of $C(\A_\Q)^{\Br_{\Upsilon,S_{\min}}(C)}$ using our implementation of Algorithm~\ref{alg1} is faster than \verb+TwoCoverDescent+, though much of the disparity likely comes down to specifics of the implementation.

\begin{Remark}
	We have compared our implementation of Algorithm~\ref{alg1} with \verb+TwoCoverDescent+ in this way for many thousands of curves, finding the outputs to be consistent in all cases when both terminated (i.e., without error, running out of memory, or being interrupted intentionally or otherwise). We take this as strong evidence of the correctness of these algorithms and their implementations in magma.
\end{Remark}

\subsection{A genus $50$ example}

\begin{Proposition}\label{prop:g50}
Let $C/\Q$ be the genus $50$ hyperelliptic curve defined by $y^2 = 5f(x)$ where
\begin{align*}
f(x) =&\phantom{,} x^{102} + x^{101} + x^{97} + x^{95} + x^{93} + x^{90} + x^{86} + x^{80} + x^{77} + x^{75} + x^{71}\\
      & + x^{70} + x^{68} + x^{65} + x^{64} + x^{63} + x^{62} + x^{59} + x^{58} + x^{53} + x^{50} + x^{49}\\ 
      &+ x^{48} + x^{46} + x^{45} + x^{44} + x^{38} + x^{37} + x^{36} + x^{35} + x^{32} + x^{31} + x^{26} \\ 
      &+ x^{25} + x^{22} + x^{16} + x^{11} + x^8 + x^7 + x + 1\,.
\end{align*}
Then $C(\A_\Q) \ne \emptyset$, but $C(\A_\Q)^{\Br(C)} = \emptyset$.
\end{Proposition}

\begin{proof}[Proof of Proposition~\ref{prop:g50}]
It is straightforward to check with the help of magma that $C$ is locally soluble. Let $\theta$ denote a root of $f(x)$. Then $L = \Q[\theta]$ is a degree $102$ number field and $\ell = \theta \in L^\times$ is an element of norm $f(0) = 1$. This corresponds to an element $\calA_\ell \in \Br_{\Upsilon}(C)$. We will show that the evaluation maps $\calA_\ell : C(\Q_v) \to \Br(\Q_v)$ are constant for all $v \in \Omega_\Q$ and are trivial if and only if $v \ne 5$. It follows that $C(\A_\Q)^{\calA_\ell} = \emptyset$. In effect we are running Algorithm~\ref{alg1} with input $\ell_1 = \ell = \theta$.

For all odd primes $v$, the valuation of the discriminant of $f(x)$ is at most $1$. So as in the proof of Lemma~\ref{lem:Aelladelic} the evaluation maps are trivial for $v \notin \{ 2,5,\infty\}$. To determine the evaluation map at $v = \infty$, we check that $f(x)$ has two real roots $r_1\le r_2 \in \R$, so $C(\R)$ is a single component whose image in $\PP^1(\R)$ is the complement of the interval $(r_1,r_2)$. Since the evaluation map is locally constant, it is constant. Taking $P \in C(\R)$ with $x(P) = a > r_2$ we have $\inv_\infty(\calA_\ell(P)) = (x(P)-r_1,r_1)_\R + (x(P)-r_2,r_2)_\R$ by Lemma~\ref{lem:HS}. Both terms in this sum are trivial because $x(P) - r_i > 0$ for $i = 1,2$.

The polynomial $f(x)$ has a single $5$-adic root $r \in \Q_5$ and it is a unit congruent to $3 \bmod 5$. Using the algorithm of \cite[Section 4]{BruinStoll} one checks that $\mu_5 : C(\Q_5) \to L_5^\times/\Q_5^{\times}L_5^{\times 2}$ is constant (because all $5$-adic points lie in a small neighborhood of the point $Q = (r,0)$). Write $5f(x) = (x-r)\tilde{f}(x)$ and let $\theta_1,\dots,\theta_n \in L_5 = \Q_5 \times \Q_5(\theta_1) \times \cdots \times \Q_5(\theta_n)$ be roots of the irreducible factors of $\tilde{f}$. Using Lemma~\ref{lem:HS} we have $\inv_5(\calA_\ell(Q)) = (r,\tilde{f}(r))_{\Q_5} + \Sigma_{\theta_i \ne r}(\theta_i,r-\theta_i)_{\Q_5(\theta_i)}$. The terms $(\theta_i,r-\theta_i)$ are all $0$ because $5$ does not divide the discriminant of $f(x)$ and this implies that any root or difference of two roots must have even valuation. On the other hand, $r \in \Z_5^{\times}$ is not a square and $\tilde{f}(r) \in 5\Z_5^\times$, so $(r,\tilde{f}(r))_{\Q_5} = 1/2$.

Using the algorithm of \cite[Section 4]{BruinStoll} we compute that $\mu_2 : C(\Q_2) \to L_2^\times/\Q_2^\times L_2^{\times 2}$ has image equal to the classes represented by $4 - \theta$, and $1/a - \theta$ for $a \in \{4,12,20,28\}$. As above, Lemma~\ref{lem:HS} reduces checking that $\calA_\ell(C(\Q_2)) = 0 \in \Br(\Q_2)$ to the computation of Hilbert symbols of these values with $\ell = \theta$ at the primes of $L$ above $2$.
\end{proof}

Let us make some further remarks on the example in Proposition~\ref{prop:g50}. The Galois group of $f(x)$ is the full symmetric group so computation of class and unit group of the number field it defines are completely out of reach even assuming GRH. Second, the Jacobian $J$ of the curve is absolutely simple and $\End(J) = \Z$ (one checks that $J_{\F_3}$ and $J_{\F_7}$ are both absolutely simple and their endomorphism algebras are linearly disjoint over $\Q$ by computing zeta functions), so there are no maps to lower positive genus curves or isogenies of degree a small power of $2$ that might be used to disprove the existence of rational points on $C$.

The curve was selected as follows. We wanted a polynomial with equal leading and constant coefficients, so that a root gives an element of norm $1$ and thus a Brauer class on the curve which might yield an obstruction to the Hasse principle. To carry out the required computations we must be able to factor the discriminant of the polynomial and compute Hilbert symbols in the completions of the number field it defines. Even the latter becomes rather expensive when there are $2$-adic primes of large degree (though a better implementation could likely improve upon this). For this reason we generated $f(x)$ as a product of minimal polynomials of random elements of $\F_{2^8}$. The curve defined by $f(x)$ has rational points above $0,\infty \in \PP^1(\Q)$, so we considered quadratic twists. The quadratic twists by $d = -1,2,3$ are not locally soluble, but the twist by $d = 5$ is. 

We carried out similar computations for other polynomials of similar degree with a norm $1$ root and found that the root gives an obstruction with reasonable frequency (at least for these high genus curves).

\subsection{Statistics from random samples}
To test the effectiveness of our algorithm, we attempted to decide on the existence of rational points on several large random samples of hyperelliptic curves over $\Q$. In this subsection we report on the results.

For various values of $g$ and $n$ we drew samples from the set $M(g,n)$ of genus $g$ hyperelliptic curves over $\Q$ defined by $y^2=f(x)$, where $f(x)=f_{2g+2}x^{2g+2}+\ldots+f_0\in \Z[x]$ is an irreducible polynomial of degree $2g+2$, and $|f_i|\leq n$ for $i=0,\ldots,2g+2$. Samples were drawn by choosing the coefficients uniformly at random. Each curve $C$ is classified as one of the following cases:
\begin{enumerate}
\item \textbf{Not Locally Soluble}; $C(\A_\Q) = \emptyset$.
\item \textbf{Brauer-Manin Obstructed}; $C(\A_\Q)\ne \emptyset$ and $C(\A_\Q)^B=\emptyset$ for $B$ a subset of $\Br_{\Upsilon,S}(C)$ such that $C(\A_\Q)^B= C(\A_\Q)^{ \Br_{\Upsilon,S}(C)}$ conditionally on GRH and $S$ is the set including all primes of $S_{\min}$, all primes up to $10^3$ and all primes of bad reduction up to $10^4$.
\item \textbf{Has a Rational Point}; $C$ has a rational point whose $x$-coordinate is of height $\le 10^6$
\item \textbf{Undecided}; $C$ does not fall into one of the categories above.
\end{enumerate}
Note that in cases (1) - (3) the existence of rational points on $C$ has been decided. 

The graphs and tables below give the proportion of curves in each case, in each of our samples. Sample sizes for each $(g,n)$  ranged from $500$ to $2000$ and can be found in the `Undecided' column of the table below.

The data allows for a few interesting observations:
\begin{itemize}
\item For each $g$ the proportion of curves that are not locally soluble appears to remain stable as $n$ varies.
\item The vast majority of of hyperelliptic curves have a Brauer-Manin obstruction to the existence of rational points which we are able to compute using Algorithm \ref{alg1}.
\item The proportion of curves with a Brauer-Manin obstruction appears to increase with both $g$ and $n$. The dependence on $g$ is in line with the results of \cite{Bhargava}. For curves of genus $\ge 5$ we are able to decide existence of rational points for over $99$\% of curves. For genus $10$ curves we have been able to decide for every curve in the sample.
\item The proportion of undecided curves appears to decrease as the genus increases, but increases slightly as $n$ is increased. 
\end{itemize}

\vfill

\begin{table}[h!]
\begin{tabular}{|l|l|l|l|rl|l|}
\hline
$g$ & $n$ & Sample    &	$C(\Q)\neq\emptyset$	  & $C(\A_\Q)^B$ & $=\emptyset$  & Undecided \\ \cline{5-6}
  	& 	  & Size      &  & \multicolumn{1}{l|}{$C(\A_\Q)=\emptyset$} & $C(\A_\Q)\neq\emptyset$  &  \\ \hline\hline
$2$ & $10$ & 1000 & $47.3\%$  & \multicolumn{1}{l|}{$17.8\%$} & $30.3\%$ &  $4.6\%$ \\ \cline{2-7} 
  & $20$  & 1000 & $36.2\%$ & \multicolumn{1}{l|}{$15.2\%$} & $38.3 \%$ & $ 10.3\%$ \\ \cline{2-7} 
  & 50    & 1000 &  $25.5\%$ & \multicolumn{1}{l|}{$15.5\%$} & $ 46.4\%$ &$ 12.6\%$  \\ \cline{2-7} 
  & 100   & 1000 &  $20.3\%$ & \multicolumn{1}{l|}{$15.8\%$} & $50.4 \%$ & $ 13.5\%$ \\ \cline{2-7} 
  & 200   & 1000 &  $14\% $& \multicolumn{1}{l|}{$15.6\%$} & $54.9 \%$ & $ 15.5\%$ \\ \cline{2-7} 
  & 500   & 1000 &  $10.3\%$& \multicolumn{1}{l|}{$18\%$} & $53.8 \%$ & $ 17.9\%$ \\ \cline{2-7} 
  & 1000  & 1000 &  $5.6\%$ & \multicolumn{1}{l|}{$15.1\%$} &  $ 59.7 \%$& $ 19.6\%$ \\ \cline{2-7} 
  & 10000 & 560 &  $ 2.7\%$& \multicolumn{1}{l|}{$16.6\%$} & $60.7\%$  &  $20\%$\\ \hline\hline

3 & 10    & 1012 & $40.1 \%$ & \multicolumn{1}{l|}{$ 17.1\%$ } & $38.8 \%$  & $ 4\%$ \\ \cline{2-7} 
  & 20    & 1000 & $32.9 \%$  & \multicolumn{1}{l|}{$14.2 \%$} & $48.5 \%$ & $ 4.4\%$ \\ \cline{2-7} 
  & 50    & 1000 & $ 24.2\%$  & \multicolumn{1}{l|}{$13.9 \%$} &$ 56.1\%$  & $ 5.8\%$ \\ \cline{2-7} 
  & 100   & 1000 & $ 18.2\%$ & \multicolumn{1}{l|}{$12.7 \%$} & $ 61.9\%$ &$ 7.2\%$  \\ \cline{2-7} 
  & 200   & 1000 &  $11\%$ & \multicolumn{1}{l|}{$ 14.4\%$} & $ 66.0\%$ & $8.6 \%$ \\ \cline{2-7} 
  & 500   & 1000 & $ 7.1\%$ &  \multicolumn{1}{l|}{$14.3 \%$ } & $ 72.7\%$ & $5.9 \%$ \\ \cline{2-7} 
  & 1000  & 1001 &$ 4.8\%$ & \multicolumn{1}{l|}{$ 15.1\%$ } & $ 73.2\%$ & $ 6.9\%$ \\ \hline\hline
  
4 & 10    & 500 & $42.6 \%$ & \multicolumn{1}{l|}{ $15.2 \%$} &  $ 41\%$ &  $ 1.2\%$ \\ \cline{2-7} 
  & 20   & 500 & $30.2 \%$ & \multicolumn{1}{l|}{ $13.8 \%$} &  $ 54.6\%$ &  $ 1.4\%$ \\ \cline{2-7} 
  & 50    & 500 & $19.2 \%$ & \multicolumn{1}{l|}{ $13.6 \%$} &  $ 64.6\%$ &  $ 2.6\%$ \\ \cline{2-7} 
  & 100   & 500 & $17.8 \%$ & \multicolumn{1}{l|}{ $12.6 \%$} &  $ 67.6\%$ &  $ 2\%$ \\ \cline{2-7} 
  & 200   & 500 & $10.8 \%$ & \multicolumn{1}{l|}{ $12\%$} &  $ 74.8\%$ &  $ 2.4\%$ \\ \cline{2-7} 
  & 500   & 500 & $7.4 \%$ & \multicolumn{1}{l|}{ $13.2 \%$} &  $ 75.8\%$ &  $ 3.6\%$ \\ \cline{2-7} 
  & 1000  & 500 & $4.4 \%$ & \multicolumn{1}{l|}{ $13\%$} &  $ 79.2\%$ &  $ 3.4\%$\\ \hline\hline

5 & 10    & 500 & $ 37.4\%$ & \multicolumn{1}{l|}{$ 14.4 \%$} & $ 47.2\%$ & $ 1 \%$ \\ \cline{2-7} 
  & 20    & 500 & $ 31.6\%$ & \multicolumn{1}{l|}{$ 13.8 \%$} & $ 53.6\%$ & $ 1 \%$ \\ \cline{2-7} 
  & 50    & 500 & $ 21.6\%$ & \multicolumn{1}{l|}{$ 10.8 \%$} & $ 67.4\%$ & $ 0.2 \%$ \\ \cline{2-7} 
  & 100    & 2000 & $ 15.9\%$ & \multicolumn{1}{l|}{$ 13.1 \%$} & $ 70.2\%$ & $ 0.8 \%$ \\ \cline{2-7} 

  & 200    & 930 & $ 9.5\%$ & \multicolumn{1}{l|}{$ 13.3 \%$} & $ 75.9\%$ & $ 1.3 \%$ \\ \hline \hline
6 & 10    & 1000 & $ 37.8\%$ & \multicolumn{1}{l|}{$ 12.7\%$} & $ 48.7\%$ & $ 0.8\%$ \\ \cline{2-7} 
  & 20    & 500 & $ 28.4\%$ & \multicolumn{1}{l|}{$ 11.4\%$} & $ 59.8\%$ & $ 0.4\%$ \\ \cline{2-7} 
  & 50    & 982 & $ 18.5\%$ & \multicolumn{1}{l|}{$ 12.5\%$} & $ 68.8\%$ & $ 0.2\%$ \\ \hline\hline
10 & 5    & 500 & $48.2 \%$ & \multicolumn{1}{l|}{$ 10.4 \%$} & $ 41.4\%$ & $0 \%$ \\ \cline{2-7} 
  & 10    & 485 & $ 35.9\%$ & \multicolumn{1}{l|}{$ 10.7\%$} & $ 53.4\%$ & $0 \%$ \\ \hline 
\end{tabular}
\vspace{5mm}
 \caption{%
Statistics on random samples of genus $g$ and coefficient bound $n$
     }%

\end{table}

\begin{figure}[]
     \begin{center}
        \subfigure[$g=2$]{%
            \label{fig:g2}
            \includegraphics[width=0.4\textwidth]{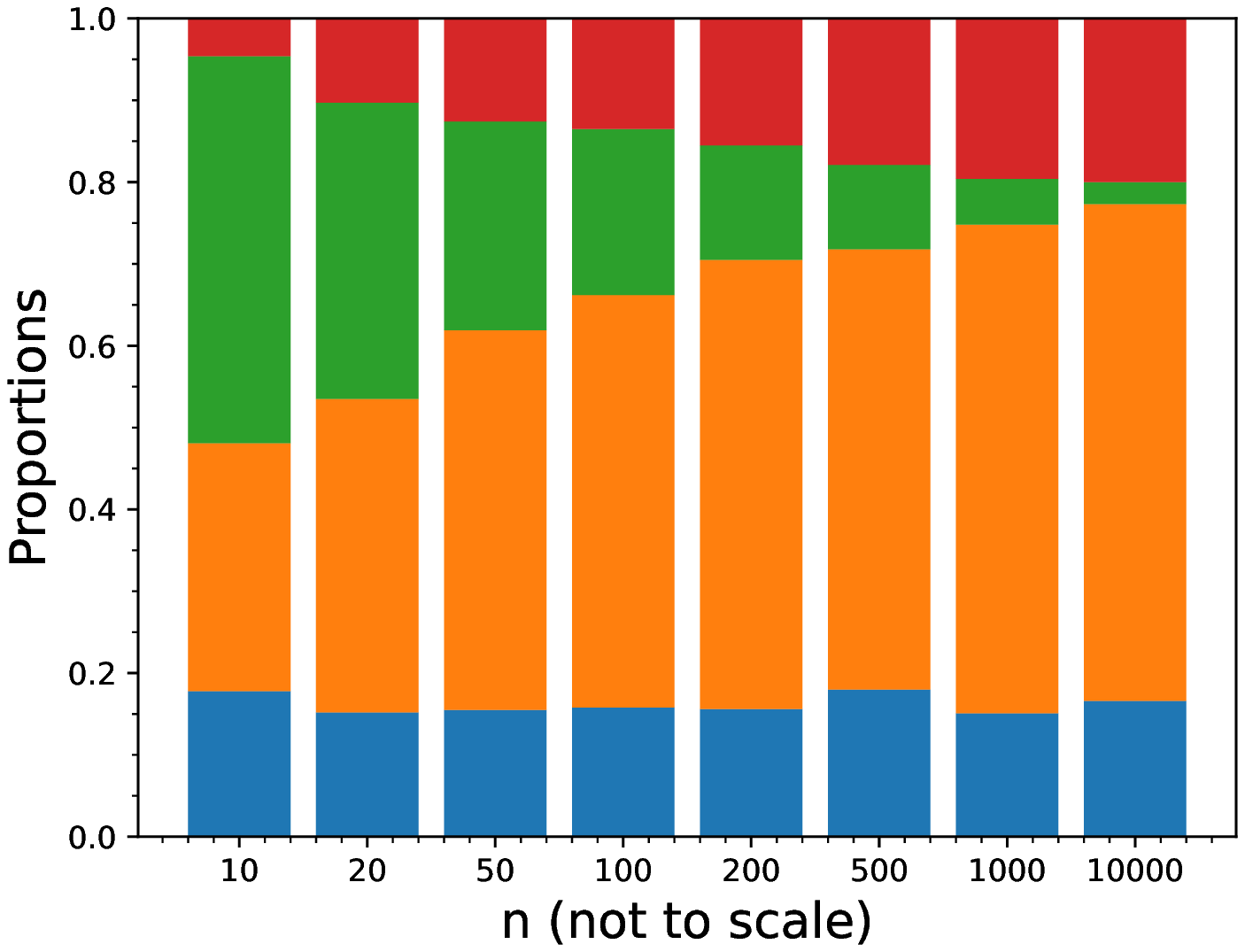}
        }%
        \subfigure[$g=3$]{%
           \label{fig:g3}
           \includegraphics[width=0.4\textwidth]{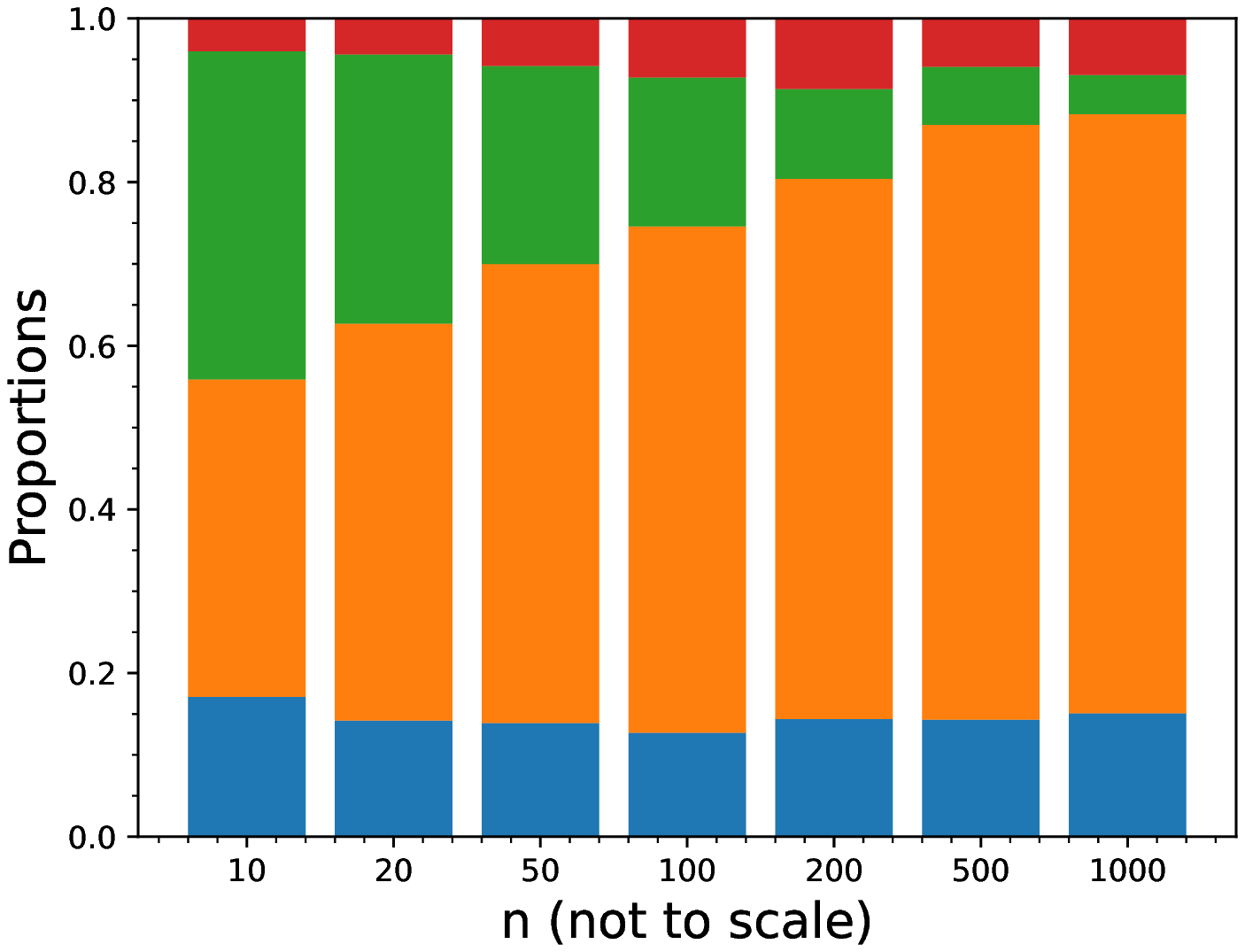}
        }\\ 
        \subfigure[$g=4$]{%
            \label{fig:g4}
            \includegraphics[width=0.4\textwidth]{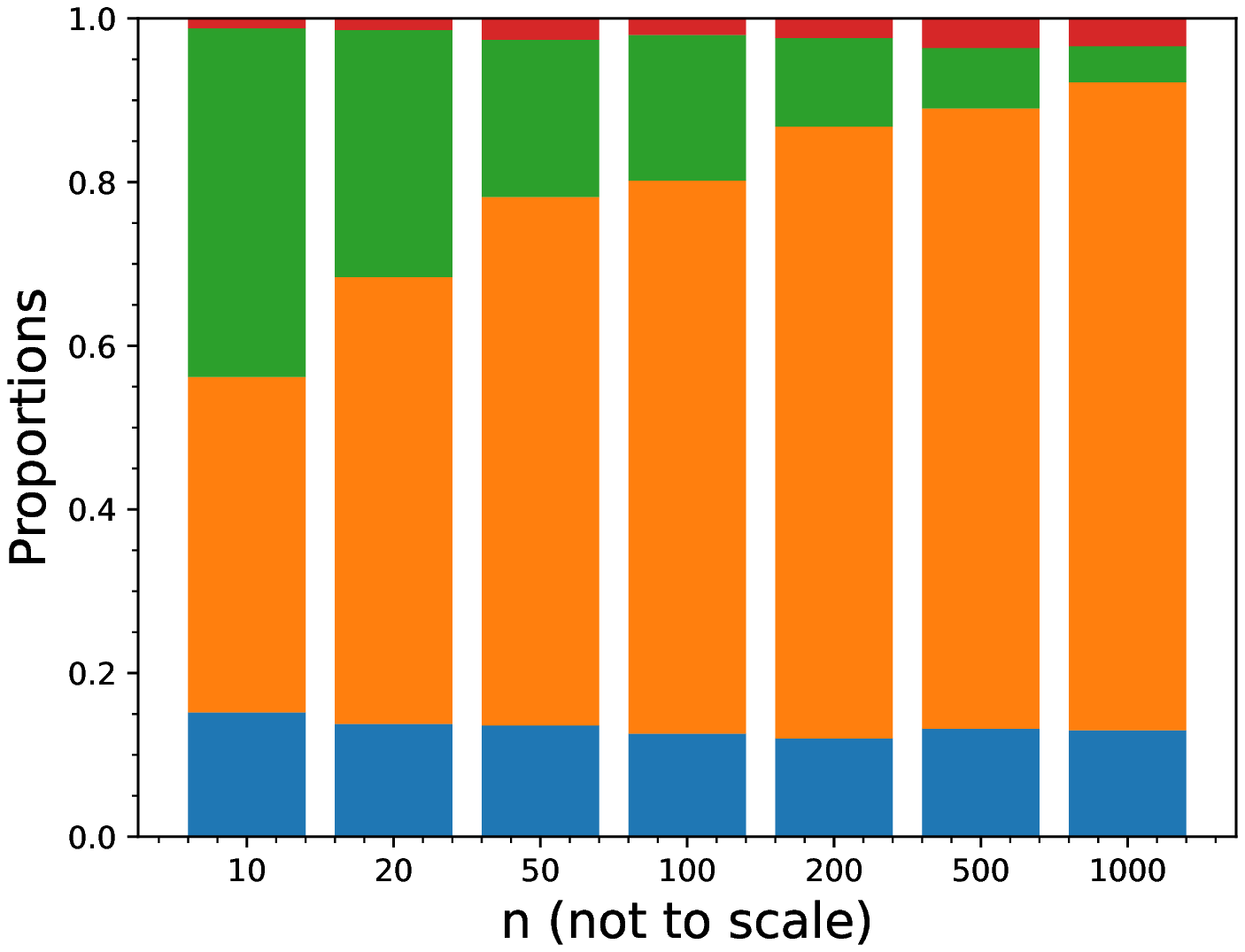}
        }%
        \subfigure[$g=5$]{%
            \label{fig:g5}
            \includegraphics[width=0.4\textwidth]{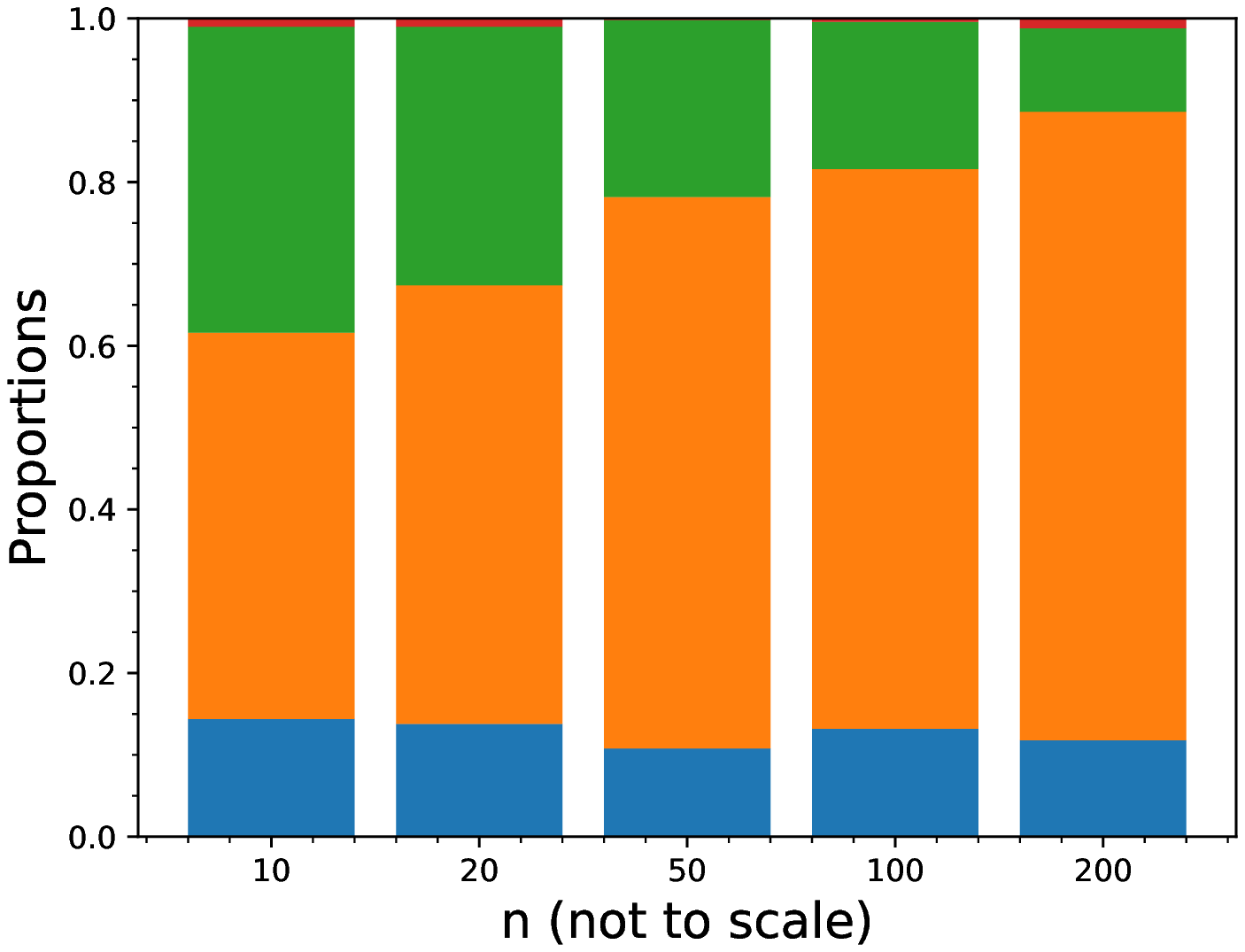}
        }%
        \\
        \subfigure[$g=6$]{%
            \label{fig:g6}
            \includegraphics[width=0.4\textwidth]{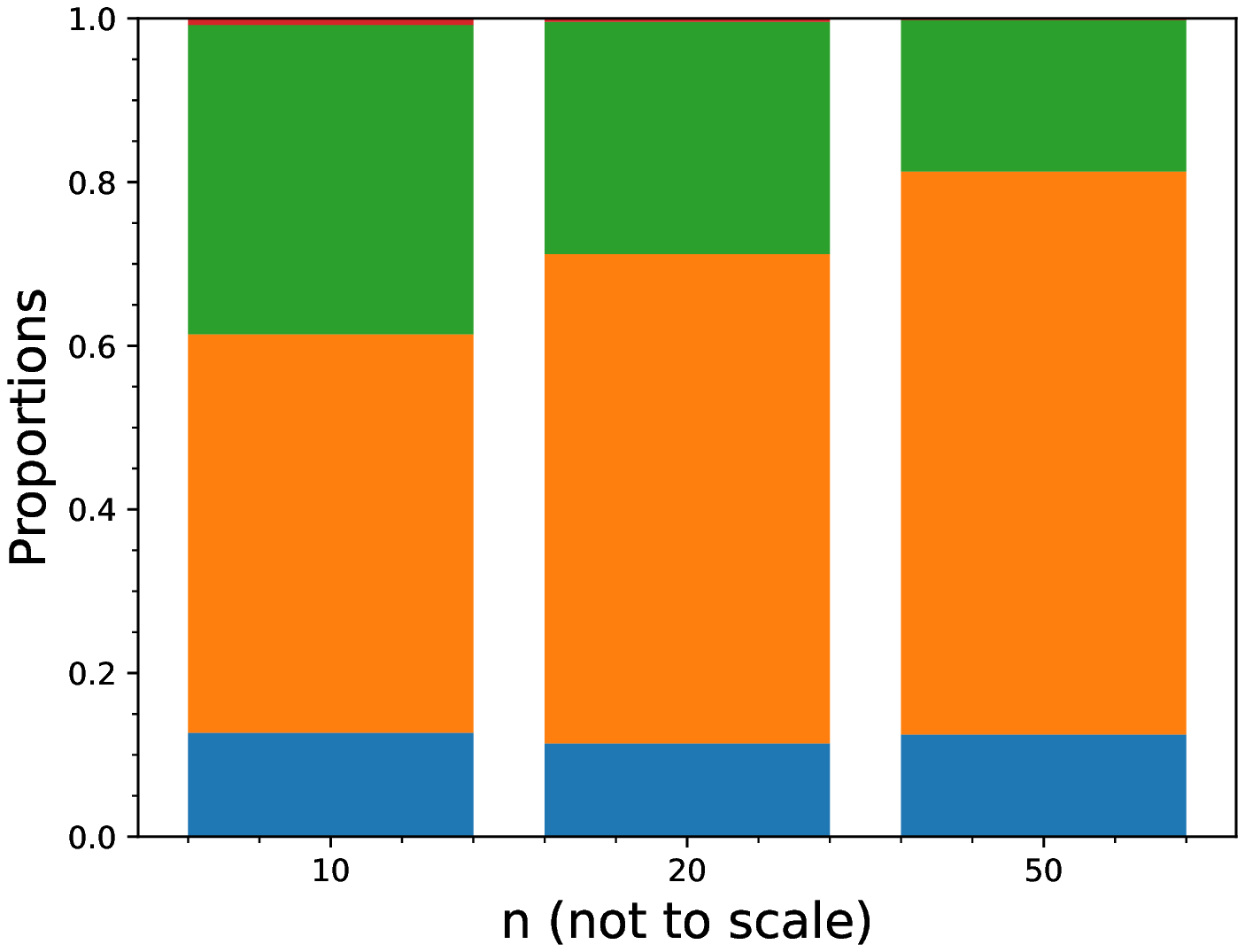}
        }%
        \subfigure[$g=10$]{%
            \label{fig:g10}
            \includegraphics[width=0.4\textwidth]{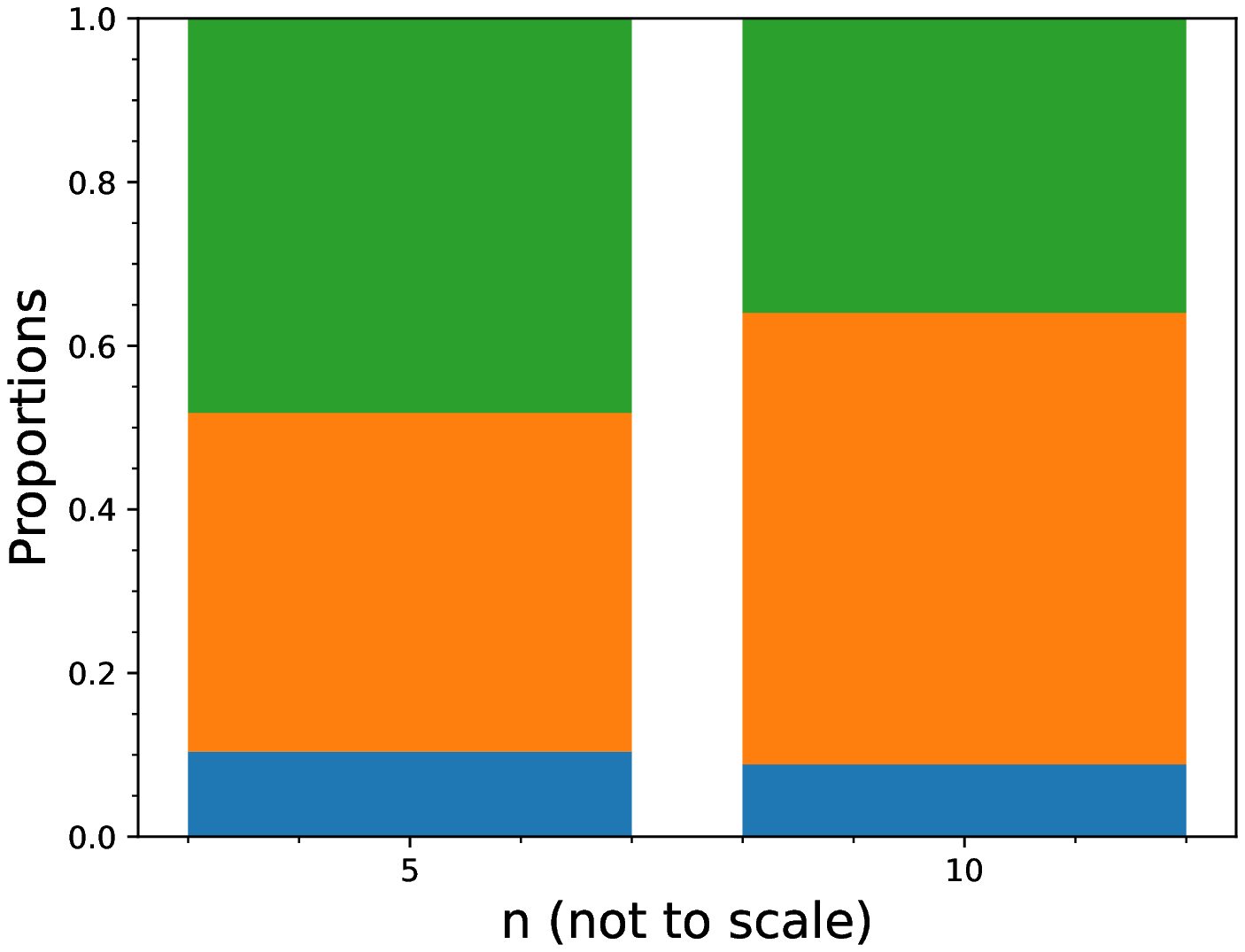}
        }%
        \\
        \subfigure[\empty]{%
            \label{fig:legend}
            \includegraphics[width=0.3\textwidth]{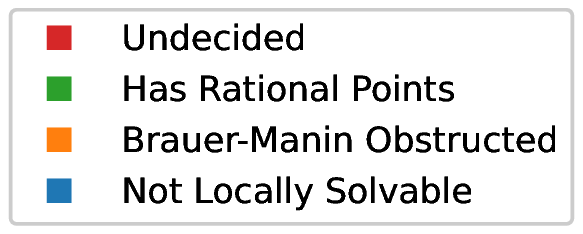}
           
        }%
    \end{center}
    \caption{%
Statistics on random samples of genus $g$ and coefficient bound $n$
     }%
   \label{fig:subfigures}
\end{figure}


\end{document}